\documentclass[12pt]{amsart}

\usepackage{amscd, amsfonts, amsmath, amssymb, ulem, amsthm, mathrsfs, appendix, multicol, enumerate, graphicx, flexisym, epsfig, float, graphicx, hyperref, pdfpages, tikz, mathtools, comment, tikz-cd}
\usepackage[all,cmtip]{xy}
\usepackage{graphicx}
\usepackage{caption}
\usepackage{subcaption}

\makeatletter
\@namedef{subjclassname@2020}{Mathematics Subject Classification \textup{2020}}

\newsavebox{\@brx}
\newcommand{\llangle}[1][]{\savebox{\@brx}{\(\m@th{#1\langle}\)}%
  \mathopen{\copy\@brx\kern-0.5\wd\@brx\usebox{\@brx}}}
\newcommand{\rrangle}[1][]{\savebox{\@brx}{\(\m@th{#1\rangle}\)}%
  \mathclose{\copy\@brx\kern-0.5\wd\@brx\usebox{\@brx}}}

\newtheorem{theorem}{Theorem}[section]
\newtheorem{corollary}[theorem]{Corollary}
\newtheorem{lemma}[theorem]{Lemma}
\newtheorem{proposition}[theorem]{Proposition}
\theoremstyle{definition}

\newtheorem{remark}[theorem]{Remark}

\numberwithin{equation}{subsection}
\newtheorem*{ack}{Acknowledgement}
\usepackage[all,cmtip]{xy}

\newcommand{\Aut}{\operatorname{Aut}}

\newcommand{\M}{\operatorname{M}}
\newcommand{\Oo}{\operatorname{O}}

\newcommand{\Z}{\operatorname{Z}}

\begin{document}

\title[Virtual extensions of symmetric groups]{Commutator subgroups and crystallographic quotients of virtual extensions of symmetric groups}

\author{Pravin Kumar}
\address{Department of Mathematical Sciences, Indian Institute of Science Education and Research (IISER) Mohali, Sector 81,  S. A. S. Nagar, P. O. Manauli, Punjab 140306, India.}
\email{pravin444enaj@gmail.com}

\author{Tushar Kanta Naik}
\address{School of Mathematical Sciences, National Institute of Science Education and Research, Bhubaneswar, An OCC of Homi Bhabha National Institute, P. O. Jatni, Khurda 752050, Odisha, India.}
\email{tushar@niser.ac.in}

\author{Neha Nanda}
\address{Laboratoire de Math\'{e}matiques Nicolas Oresme UMR CNRS 6139, Universit\'{e} de Caen Normandie, 14000 Caen, France.}
\email{nehananda94@gmail.com}

\author{Mahender Singh}
\address{Department of Mathematical Sciences, Indian Institute of Science Education and Research (IISER) Mohali, Sector 81,  S. A. S. Nagar, P. O. Manauli, Punjab 140306, India.}
\email{mahender@iisermohali.ac.in}


\begin{abstract}
The virtual braid group $VB_n$, the virtual twin group $VT_n$ and the virtual triplet group $VL_n$ are extensions of the symmetric group $S_n$, which are motivated by the Alexander-Markov correspondence for virtual knot theories. The kernels of natural epimorphisms of these groups onto the symmetric group $S_n$ are the pure virtual braid group $VP_n$, the pure virtual twin group $PVT_n$ and the pure virtual triplet group $PVL_n$, respectively. In this paper, we investigate commutator subgroups, pure subgroups and crystallographic quotients of these groups. We derive explicit finite presentations of the pure virtual triplet group $PVL_n$,  the commutator subgroup $VT_n^{'}$ of $VT_n$ and the commutator subgroup $VL_n^{'}$  of $VL_n$. Our results complete the understanding of these groups, except that of $VB_n^{'}$, for which the existence of a finite presentations is not known for $n \ge 4$. We also prove that $VL_n/PVL_n^{'}$ is a crystallographic group and give an explicit construction of infinitely many torsion elements in it.
\end{abstract}

\maketitle

\section{Introduction}
For  $n \ge 2$, the  symmetric group $S_n$ admits a Coxeter presentation with generators $\{\tau_1, \ldots, \tau_{n-1} \} $ and defining relations:
\begin{enumerate}
\item $\tau_i^2 = 1$ for  $1\leq i \leq n-1$,
\item $\tau_i\tau_j\tau_i=\tau_j\tau_i\tau_j$ for $|i-j|=1$,
\item $\tau_i\tau_j=\tau_j\tau_i$ for $\mid i - j\mid \geq 2$.
\end{enumerate}
\par

By omitting all the relations of type (1), (2) or (3) one at a time from the preceding presentation, we obtain three canonical extensions of $S_n$, namely, the braid group $B_n$,  the twin group $T_n$ and the triplet group $L_n$.

\begin{Small}
\begin{eqnarray}
 & & B_n =\Bigg\langle \sigma_1, \sigma_2, \ldots, \sigma_{n-1} ~\mid~ \sigma_i\sigma_j=\sigma_j\sigma_i \quad \textrm{for} \quad |i-j| \ge 2, \quad \sigma_i\sigma_j\sigma_i=\sigma_j\sigma_i\sigma_j \quad \text{for} \quad |i-j|=1 \Bigg\rangle.\\
&&  T_n = \Bigg\langle s_1, s_2, \ldots, s_{n-1} ~\mid~ s_i^2 = 1  \quad \textrm{for} \quad 1 \le i\le n-1, \quad s_is_j=s_js_i \quad  \text{for} \quad |i - j|  \geq 2 \Bigg\rangle.\\
 &&  L_n= \Bigg\langle y_1, y_2, \ldots, y_{n-1} ~\mid~ y_i^2 = 1  \quad \textrm{for} \quad 1 \le i\le n-1, \quad y_iy_jy_i=y_jy_i y_j \quad \text{for} \quad  |i-j| = 1\Bigg\rangle.
\end{eqnarray}
\end{Small}

The terms twin groups and triplet groups first appeared in the work \cite{MR1386845} of Khovanov on $K(\pi,1)$ subspace arrangements. The kernels of natural epimorphisms
$$B_n \to S_n \quad \textrm{given by} \quad  \sigma_i \mapsto \tau_i,$$
$$T_n \to S_n \quad \textrm{given by} \quad  s_i \mapsto \tau_i$$
and
$$L_n \to S_n \quad \textrm{given by} \quad  y_i \mapsto \tau_i$$
are the pure braid group $P_n$, the pure twin group $PT_n$ and the pure triplet group $PL_n$, respectively.  Pure braid groups are well-known to have complements of complex hyperplane arrangements as their Eilenberg--MacLane spaces. Pure twin groups and pure triplet groups are real counterparts of pure braid groups, and also admit suitable codimension two subspace arrangements as their Eilenberg--MacLane spaces \cite{MR1386845}.  Furthermore, these groups have topological interpretations similar to braid groups, and play the role of groups under the Alexander-Markov correspondence for some appropriate knot theories. Twin groups are related to doodles on the 2-sphere, which were introduced by Fenn and Taylor \cite{MR0547452} as finite collections of simple closed curves on the 2-sphere without triple or higher intersections. Allowing self intersections of curves, Khovanov \cite{MR1370644} extended the idea to finite collections of closed curves without triple or higher intersections on a closed oriented surface. Similarly, triplet groups are related to topological objects called noodles. Fixing a codimension one foliation (with singular points) on the 2-sphere, a noodle is a collection of closed curves on the 2-sphere such that no two intersection points belong to the same leaf of the foliation, there are no quadruple intersections and no intersection point can occupy a singular point of the foliation. 
\par

Khovanov \cite{MR1370644} proved an analogue of Alexander Theorem by showing that every oriented doodle on a $2$-sphere is closure of a twin. An analogue of Markov Theorem in this setting has been proved by Gotin \cite{Gotin}, though the idea has been implicit in \cite{Khovanov1990}. The structure of the pure twin group $PT_n$ for small values of $n$ has been well-understood \cite{MR4027588, MR4170471, MR4079623}, but understanding the general case seems challenging (see \cite{Mostovoy2020} for some progress). Concerning triplet groups, it has been shown in \cite{MR4270786} that the pure triplet group $PL_n$ is a non-abelian free group of finite rank for $n \ge 4$.
\par

The virtual knot theory is a far reaching generalisation of the classical knot theory and was introduced by Kauffman in \cite{MR1721925}. Motivated by the well-known Alexander-Markov correspondence for  virtual knot theories, the preceding three extensions of $S_n$ can be enlarged to their virtual avatars, namely, the virtual  braid group $VB_n$ \cite{MR2128049},  the virtual  twin group $VT_n$ \cite{MR4027588, MR4209535} and  the virtual  triplet  group $VL_n$ (which we introduce below).

\begin{small}
\begin{eqnarray}
  &VB_n=& \Bigg\langle \sigma_1, \sigma_2, \ldots, \sigma_{n-1}, \rho_1, \rho_2, \ldots, \rho_{n-1}~\mid~ \rho_i^2=1 \quad \text{for} \quad1\leq i \leq n-1,\\
\nonumber &&\sigma_i \sigma_j= \sigma_j \sigma_i, \quad \rho_i\rho_j=\rho_j \rho_i, \quad  \sigma_i\rho_j= \rho_j\sigma_i  \quad \text{for} \quad |i-j| \ge 2,\\
\nonumber &&\sigma_i\sigma_j\sigma_i=\sigma_j\sigma_i\sigma_j, \quad \rho_i\rho_j\rho_i=\rho_j\rho_i\rho_j, \quad  \rho_i\sigma_j\rho_i=\rho_j\sigma_i\rho_j \quad \text{for} \quad |i-j| =1\Bigg\rangle.
\end{eqnarray}
\end{small}

\begin{small}
\begin{eqnarray}\label{vtn presentation}
  &VT_n=& \Bigg\langle s_1, s_2, \ldots, s_{n-1}, \rho_1, \rho_2, \ldots, \rho_{n-1}~\mid~ s_i^2=1, \quad \rho_i^2=1 \quad \text{for} \quad1\leq i \leq n-1,\\
\nonumber &&s_i s_j= s_j s_i, \quad \rho_i\rho_j=\rho_j \rho_i, \quad  s_i\rho_j= \rho_j s_i  \quad \text{for} \quad |i-j| \ge 2,\\
\nonumber && \rho_i\rho_j\rho_i=\rho_j\rho_i\rho_j, \quad  \rho_i s_j\rho_i=\rho_js_i\rho_j \quad \text{for} \quad |i-j| =1\Bigg\rangle.
\end{eqnarray}
\end{small}
\begin{small}
\begin{eqnarray}
  &VL_n=& \Bigg\langle y_1, y_2, \ldots, y_{n-1}, \rho_1, \rho_2, \ldots, \rho_{n-1}~\mid~ y_i^2=1, \quad \rho_i^2=1 \quad \text{for} \quad1\leq i \leq n-1,\\
\nonumber && \rho_i\rho_j=\rho_j \rho_i, \quad  y_i\rho_j= \rho_j y_i  \quad \text{for} \quad |i-j| \ge 2,\\
\nonumber &&y_iy _jy_i=y_jy_iy_j, \quad \rho_i\rho_j\rho_i=\rho_j\rho_i\rho_j, \quad  \rho_iy_j\rho_i=\rho_jy_i\rho_j \quad \text{for} \quad |i-j| =1\Bigg\rangle.
\end{eqnarray}
\end{small}
There are natural epimorphisms 
$$VB_n \to S_n \quad \textrm{given by} \quad  \sigma_i, \rho_i \mapsto \tau_i,$$
$$VT_n \to S_n \quad \textrm{given by} \quad  s_i, \rho_i \mapsto \tau_i$$
and
$$VL_n \to S_n \quad \textrm{given by} \quad  y_i, \rho_i \mapsto \tau_i.$$
The kernels of these epimorphisms are the pure virtual braid group $VP_n$, the pure virtual twin group $PVT_n$ and the pure virtual triplet group $PVL_n$, respectively. The virtual groups together with their pure subgroups make the following diagram commute.
\begin{Small}
\[\begin{tikzcd}
	{PVT_n} && {S_n} && {PVL_n} \\
	& {VT_n} && {VL_n} \\
	&& {S_n} \\
	& {S_n} && {S_n} \\
	&& {VB_n} \\
	&& {VP_n}
	\arrow[hook, from=6-3, to=5-3]
	\arrow[Rightarrow, no head, from=1-3, to=3-3]
	\arrow[Rightarrow, no head, from=3-3, to=4-2]
	\arrow[Rightarrow, no head, from=3-3, to=4-4]
	\arrow[two heads, from=2-2, to=3-3]
	\arrow[two heads, from=2-4, to=3-3]
	\arrow[two heads, from=5-3, to=3-3]
	\arrow[hook, from=4-4, to=5-3]
	\arrow[hook', from=4-2, to=5-3]
	\arrow[hook, from=4-4, to=2-4]
	\arrow[hook, from=4-2, to=2-2]
	\arrow[hook, from=1-1, to=2-2]
	\arrow[hook, from=1-5, to=2-4]
	\arrow[hook, from=1-3, to=2-2]
	\arrow[hook, from=1-3, to=2-4]
\end{tikzcd}\]
\end{Small}

The present paper is concerned with commutator subgroups, pure subgroups  and certain geometrically motivated quotients of virtual braid groups, virtual twin groups and virtual triplet groups. It is known from \cite[Theorem 1.1]{MR3955820} that the commutator subgroup $VB_n^{'}$ of the virtual braid group $VB_n$ is generated by $(2n- 3)$ elements for $n \ge 4$, whereas $VB_3^{'}$ is infinitely generated. It is still not known whether $VB_n^{'}$ is finitely presented for $n\ge4$. In this paper, we determine a precise finite presentation of the commutator subgroup $VT_n^{'}$ of $VT_n$
(Theorem \ref{commutator-subgroup-vtn}) and the commutator subgroup $VL_n^{'}$ of $VL_n$ (Theorem \ref{commutator-subgroup-vln}). A finite presentation of the pure virtual braid group is well-known due to Bardakov \cite{MR2128039}. It has been proved in \cite{NaikNandaSingh2} that the pure virtual twin group $PVT_n$ is an irreducible right-angled Artin group and a finite presentation has also been given for the same. We use the Reidemeister-Schreier method to derive a finite presentation of the pure virtual triplet group $PVL_n$ (Theorem \ref{presentation pure virtual triplet group}). This completes our understanding of these groups in the virtual setting. As by-products, we deduce that $VT_n$ and $VL_n$ are residually nilpotent if and only if $n = 2$ (Corollaries \ref{vtn residually nilpotent} and \ref{vln residually nilpotent}). 
\par
Crystallographic groups play a crucial role in the study of  groups of isometries of Euclidean spaces. In fact, there is a correspondence between torsion free crystallographic groups (called Bieberbach groups) and the class of compact flat Riemannian manifolds \cite[Theorem 2.1.1]{MR1482520}. Crystallographic quotients of virtual braid groups and virtual twin groups have been investigated in \cite{Ocampo-Santos-2021}. In this paper, we consider the case of virtual triplet groups. More precisely, we prove that $VL_n / PVL_n^{'}$ is a crystallographic group of dimension $n(n-1)/2$~ for $n \geq 2$ (Theorem \ref{cryst. quotient virtual triplet mod pure}), and that it contains infinitely many torsion elements (Corollary \ref{torsion in vln mod pvln}). In fact, we give an explicit construction of  infinitely many torsion elements in $V L_n / PV L_n^{'}$.
\par

We denote the $n$-th term of the lower central series of a group $G$ by $\gamma_n(G)$, where $\gamma_1(G)=G$ and $\gamma_2(G)=G^{'}$ is the commutator subgroup of $G$. 

\medskip

\section{Presentation of commutator subgroup of virtual twin group}\label{sec commutator subgroup of virtual twin}
In this section,  we give a finite presentation of the commutator subgroup $VT_n^{'}$ of $VT_n$. This is achieved through a reduced presentation of $VT_n$, which we prove first.

\begin{theorem}\label{reduced-presentation-n4}
The virtual twin group $VT_n$ $(n \ge 3)$ has the following reduced presentation:
\begin{enumerate}
\item $VT_3 =\langle  s_1, \rho_1, \rho_2 ~\mid~  s_1^2 = \rho_1^2 =\rho_2^2 = (\rho_1\rho_2)^3=1 \rangle$.
\item For $n\geq 4$, $VT_n$ has a presentation with generating set $\{s_1, \rho_1, \rho_2,\dots,\rho_{n-1}\}$ and defining relations as follows:
\begin{align} 
s_1^{2} &=1, \label{nvtn 8} \\
\rho_i^{2} &= 1 \quad \textrm{for} \quad 1 \le i \le n-1,\label{nvtn 9} \\
\rho_i\rho_j &= \rho_j\rho_i \quad \textrm{for} \quad |i - j| \geq 2,\label{nvtn 10} \\
\rho_i\rho_{i+1}\rho_i &= \rho_{i+1}\rho_i\rho_{i+1} \quad \textrm{for} \quad 1 \le i \le n-2,\label{nvtn 11} \\
\rho_is_1 &= s_1\rho_i \quad \textrm{for} \quad i \geq 3,\label{nvtn 12} \\
(s_1\rho_2 \rho_1\rho_3 \rho_2)^4 &= 1. \label{nvtn 13}
\end{align}
\end{enumerate}
\end{theorem}

\begin{proof}
The case $n=3$ is immediate. For $n \ge 4$, we first show that the generating set of $VT_n$ can be reduced to $\{s_1, \rho_1, \rho_2,\dots,\rho_{n-1}\}$.
\medskip

Claim 1:  $s_{i+1}=(\rho_i \rho_{i-1}\dots \rho_2\rho_1)(\rho_{i+1} \rho_{i}\dots \rho_3\rho_2)s_1(\rho_{2} \rho_{3}\dots \rho_i\rho_{i+1})(\rho_1 \rho_2\dots \rho_{i-1}\rho_i)$ for $i \geq 1$.
\par
The case $i=1$ follows from the relation 
\begin{equation}\label{vtn mixed}
 \rho_i s_j\rho_i=\rho_js_i\rho_j \quad \text{for} \quad |i-j| =1.
\end{equation} 
Assuming the claim for $i$ and using the relation \eqref{vtn mixed} for $j=i+1$ gives
$$s_{i+1}=\rho_i \rho_{i+1}(\rho_{i-1} \rho_{i-2}\dots \rho_2\rho_1)(\rho_{i} \rho_{i-1}\dots \rho_3\rho_2)s_1(\rho_{2} \rho_{3}\dots \rho_{i-1}\rho_{i})(\rho_1 \rho_2\dots \rho_{i-2}\rho_{i-1})\rho_{i+1} \rho_{i}.$$
Now, the relation \eqref{nvtn 10} gives
\begin{equation}\label{claim 1 eq}
s_{i+1}=(\rho_i \rho_{i-1}\dots \rho_2\rho_1)(\rho_{i+1} \rho_{i}\dots \rho_3\rho_2)s_1(\rho_{2} \rho_{3}\dots \rho_i\rho_{i+1})(\rho_1 \rho_2\dots \rho_{i-1}\rho_i),
\end{equation}
which is desired. Hence, we can eliminate the generators $s_i$ for $i \geq 2$.
\medskip

Next, we show that the relations of $VT_n$ involving $s_i$ for $i \geq 2$ can be recovered from the relations listed in the statement of the theorem.
\medskip

Claim 2:  The relations $\rho_i s_{i+1}\rho_i=\rho_{i+1} s_{i}\rho_{i+1}$ for $1 \le i \le n-2$ can be recovered from the relations \eqref{nvtn 9} and \eqref{nvtn 10}.
\par 
Using \eqref{claim 1 eq}, we have
\begin{eqnarray*}
\rho_i s_{i+1}\rho_i  &=& \rho_i (\rho_{i} \ldots \rho_1)(\rho_{i+1} \ldots \rho_2) s_1 (\rho_2 \ldots \rho_{i+1})(\rho_1 \ldots \rho_{i})\rho_{i}\\
&=& (\rho_{i-1} \ldots \rho_1)\rho_{i+1}(\rho_i \ldots \rho_2) s_1 (\rho_2 \ldots \rho_{i})\rho_{i+1}(\rho_1 \ldots \rho_{i-1})\\
&=& \rho_{i+1}(\rho_{i-1} \ldots \rho_1)(\rho_i \ldots \rho_2) s_1 (\rho_2 \ldots \rho_{i})(\rho_1 \ldots \rho_{i-1})\rho_{i+1}\\
&=& \rho_{i+1}s_{i}\rho_{i+1},
\end{eqnarray*}
which is desired.
\medskip

Claim 3:  The relations $s_i \rho_j = \rho_j s_i$ for $|i-j| \geq 2$ can be recovered using the relations \eqref{nvtn 10}, \eqref{nvtn 11} and \eqref{nvtn 12}.
\par
If  $j \leq i-2$, then we have
\begin{Small}
  \begin{eqnarray*}
&& s_i \rho_j \\
&=& (\rho_{i-1} \rho_{i-2}\dots \rho_2\rho_1)(\rho_{i} \rho_{i-1}\dots \rho_3\rho_2)s_1(\rho_{2} \rho_{3}\dots \rho_{i-1}\rho_{i})(\rho_1 \rho_2\dots \rho_{i-2}\rho_{i-1})\rho_j\\
&=& (\rho_{i-1} \rho_{i-2}\dots \rho_2\rho_1)(\rho_{i} \rho_{i-1}\dots \rho_3\rho_2)s_1(\rho_{2} \dots \rho_{i-1}\rho_{i})(\rho_1 \rho_2\dots \rho_j \rho_{j+1} \rho_j \dots \rho_{i-1}) ~ \text{ (By \eqref{nvtn 10})}\\
&=& (\rho_{i-1} \rho_{i-2}\dots \rho_2\rho_1)(\rho_{i} \rho_{i-1}\dots \rho_3\rho_2)s_1(\rho_{2} \dots \rho_{i-1}\rho_{i})(\rho_1 \dots \rho_{j+1} \rho_{j} \rho_{j+1} \dots \rho_{i-1}) ~ \text{ (By \eqref{nvtn 11})}\\
&=& (\rho_{i-1} \rho_{i-2}\dots \rho_2\rho_1)(\rho_{i} \rho_{i-1}\dots \rho_3\rho_2)s_1(\rho_{2} \dots \rho_{j+1} \rho_{j+2} \rho_{j+1} \dots \rho_{i})(\rho_1 \dots \rho_{i-1}) ~ \text{ (By \eqref{nvtn 10})}\\
&=& (\rho_{i-1} \rho_{i-2}\dots \rho_2\rho_1)(\rho_{i} \rho_{i-1}\dots \rho_3\rho_2)s_1(\rho_{2} \dots \rho_{j+2} \rho_{j+1} \rho_{j+2} \dots \rho_{i})(\rho_1 \dots \rho_{i-1}) ~ \text{ (By \eqref{nvtn 11})}\\
&=& (\rho_{i-1} \rho_{i-2}\dots \rho_2\rho_1)(\rho_{i} \dots \rho_{j+2} \rho_{j+1} \rho_{j+2} \dots \rho_2)s_1(\rho_{2} \dots \rho_{i})(\rho_1 \dots \rho_{i-1}) ~ \text{ (By \eqref{nvtn 10} and \eqref{nvtn 12})}\\
&=& (\rho_{i-1} \rho_{i-2}\dots \rho_2\rho_1)(\rho_{i} \dots \rho_{j+1} \rho_{j+2} \rho_{j+1} \dots \rho_2)s_1(\rho_{2} \dots \rho_{i})(\rho_1 \dots \rho_{i-1}) ~ \text{ (By \eqref{nvtn 11})}\\
&=& (\rho_{i-1} \dots \rho_{j+1} \rho_{j}\rho_{j+1} \dots \rho_1)(\rho_{i} \dots  \rho_2)s_1(\rho_{2} \dots \rho_{i})(\rho_1 \dots \rho_{i-1}) ~ \text{(By \eqref{nvtn 10})}\\
&=& (\rho_{i-1} \dots \rho_{j} \rho_{j+1}\rho_{j} \dots \rho_1)(\rho_{i} \dots  \rho_2)s_1(\rho_{2} \dots \rho_{i})(\rho_1 \dots \rho_{i-1}) ~ \text{ (By \eqref{nvtn 11})}\\
&=&\rho_{j} (\rho_{i-1} \dots \rho_1)(\rho_{i} \dots  \rho_2)s_1(\rho_{2} \dots \rho_{i})(\rho_1 \dots \rho_{i-1}) ~ \text{ (By \eqref{nvtn 10})}\\
&=& \rho_j s_i,
\end{eqnarray*}
\end{Small}
which is desired. If $j \geq i+2$, then the claim follows from the relations \eqref{nvtn 10} and \eqref{nvtn 12}. 
\medskip

Claim 4: The relations $s_i s_j = s_j s_i $ for $1 \le i \le j-2 \le n-3$ can be recovered from the relations \eqref{nvtn 8} through \eqref{nvtn 13}. 
\par
For clarity, we underline the terms which are modified using the cited relations. First, we consider $i=1$, in which case we have
\begin{eqnarray*}
&&s_1 s_j \\
&=&\dashuline{s_1(\rho_{j-1} \ldots \rho_1)}(\rho_j \ldots \rho_2) s_1 (\rho_2 \dashuline{\rho_3\ldots \rho_j)(\rho_1\rho_2} \ldots \rho_{j-1}) \\
&=&(\rho_{j-1} \ldots \rho_{3})\dashuline{s_1\rho_2\rho_1(\rho_{j} \ldots \rho_3\rho_2}) s_1 \rho_2\rho_1\rho_3\rho_2 (\rho_4 \ldots \rho_j)(\rho_3 \ldots \rho_{j-1}) \\
&=&(\rho_{j-1} \ldots \rho_{3})(\rho_{j} \ldots \rho_4)(s_1 \rho_2\rho_1\rho_3\rho_2)^2 (\rho_4 \ldots \rho_j)(\rho_3 \ldots \rho_{j-1}) \\
&=&(\rho_{j-1} \ldots \rho_{3})(\rho_{j} \ldots \rho_4)(\rho_2\rho_3\rho_1\rho_2s_1)^2 (\rho_4 \ldots \rho_j)(\rho_3 \ldots \rho_{j-1}) \\
&=&(\rho_{j-1} \ldots \rho_1)(\rho_j \ldots \rho_2) s_1 (\rho_2 \dashuline{\rho_3\ldots \rho_j)(\rho_1\rho_2} \ldots \rho_{j-1})s_1 \\
&=& s_js_1.
 \end{eqnarray*}

Next, we assume that $i \ge 2$. Then we have
\begin{Small}
\begin{eqnarray*}
&&s_i s_j \\
&=&(\rho_{i-1} \ldots \rho_1)(\rho_i \ldots \rho_2) s_1(\rho_2 \ldots \rho_i)(\rho_1 \ldots \rho_{i-1})(\dashuline{\rho_{j-1} \ldots \rho_{i+2} \rho_{i+1}} \rho_i \ldots \rho_1)(\rho_j \ldots \rho_2) s_1 \\
&&(\rho_2 \ldots \rho_j)(\rho_1 \ldots \rho_{j-1}) \\
&=&(\rho_{j-1} \ldots \rho_{i+2})(\rho_{i-1} \ldots \rho_1)(\rho_i \ldots \rho_2) s_1(\rho_2 \ldots \rho_i) \rho_{i+1}\dashuline{(\rho_{1 \ldots} \rho_{i-1})(\rho_i \ldots \rho_1)}(\rho_j \ldots \rho_2) s_1   \\
&& (\rho_2 \ldots \rho_j)(\rho_1 \ldots \rho_{j-1}) ~\text{ (By \eqref{nvtn 10})} \\
&=&(\rho_{j-1} \ldots \rho_{i+2})(\rho_{i-1} \ldots \rho_1)(\rho_i \ldots \rho_2) s_1\dashuline{(\rho_2 \ldots \rho_i \rho_{i+1})(\rho_i \ldots \rho_2} \rho_1 \rho_2 \ldots \rho_i)(\rho_j \ldots \rho_2) s_1 \\
&& (\rho_2 \ldots \rho_j)(\rho_1 \ldots \rho_{j-1}) ~\text{ (By \eqref{nvtn 11})}\\
&=&(\rho_{j-1} \ldots \rho_{i+2})(\rho_{i-1} \ldots \rho_1)(\rho_i \ldots \rho_2) s_1(\dashuline{\rho_{i+1} \ldots \rho_3} \rho_2 \rho_3 \ldots \rho_{i+1})(\rho_1 \ldots \rho_i) \\
& &(\dashuline{\rho_j \ldots \rho_{i+3} \rho_{i+2}} \rho_{i+1} \ldots \rho_2) s_1(\rho_2 \ldots \rho_j)(\rho_1 \ldots \rho_{j-1}) ~\text{ (By \eqref{nvtn 11})}\\
&=&(\rho_{j-1} \ldots \rho_{i+2})(\rho_j \ldots \rho_{i+3})(\rho_{i-1} \ldots \rho_1)(\rho_i \ldots \rho_2)(\rho_{i+1} \ldots \rho_3) s_1(\rho_2 \ldots \rho_{i+1}) \rho_{i+2} \\
&& (\rho_1 \dashuline{\rho_2 \ldots \rho_i)(\rho_{i+1} \rho_i \ldots \rho_2)} s_1(\rho_2 \ldots \rho_j)(\rho_1 \ldots \rho_{j-1}) ~\text{ (By \eqref{nvtn 10} and \eqref{nvtn 12})} \\
&=&(\rho_{j-1} \ldots \rho_{i+2})(\rho_j \ldots \rho_{i+3})(\rho_{i-1} \ldots \rho_1)(\rho_i \ldots \rho_2)(\rho_{i+1} \ldots \rho_3) s_1(\rho_2 \ldots \rho_{i+2}) \\
&& \rho_1(\dashuline{\rho_{i+1} \ldots \rho_3} \rho_2 \rho_3 \ldots \rho_{i+1}) \dashuline{s_1}(\rho_2 \ldots \rho_j)(\rho_1 \ldots \rho_{j-1}) ~\text{ (By \eqref{nvtn 11})}\\
& =&(\rho_{j-1} \ldots \rho_{i+2})(\rho_j \ldots \rho_{i+3})(\rho_{i-1} \ldots \rho_1)(\rho_i \ldots \rho_2)(\rho_{i+1} \ldots \rho_3) s_1 (\rho_2 \dashuline{\rho_3 \ldots \rho_{i+2})} \\
&& \dashuline{(\rho_{i+1} \ldots \rho_3)} \rho_1 \rho_2 s_1(\rho_3 \ldots \rho_{i+1})(\rho_2 \ldots \rho_j)(\rho_1 \ldots \rho_{j-1}) ~\text{ (By \eqref{nvtn 10} and \eqref{nvtn 12})}\\
&=& (\rho_{j-1} \ldots \rho_{i+2})(\rho_j \ldots \rho_{i+3})(\rho_{i-1} \ldots \rho_1)(\rho_i \ldots \rho_2)(\rho_{i+1} \ldots \rho_3) s_1 \rho_2(\dashuline{\rho_{i+2} \ldots \rho_4} \rho_3 \dashuline{\rho_4 \ldots \rho_{i+2}}) \\
&&  \rho_1 \rho_2 s_1(\rho_3 \ldots \rho_{i+1})(\rho_2 \ldots \rho_j)(\rho_1 \ldots \rho_{j-1}) ~\text{ (By \eqref{nvtn 11})}\\
& = & (\rho_{j-1} \ldots \rho_{i+2})(\rho_j \ldots \rho_{i+3})(\rho_{i-1} \ldots \rho_1)(\rho_i \ldots \rho_2)(\rho_{i+1} \ldots \rho_3)(\rho_{i+2} \ldots \rho_4)\\
&& s_1 \rho_2 \rho_3 \rho_1 \rho_2 s_1 \dashuline{(\rho_4 \ldots \rho_{i+2})(\rho_3 \ldots \rho_{i+1})(\rho_2 \ldots \rho_j)(\rho_1 \ldots \rho_{j-1})} ~\text{ (By \eqref{nvtn 10} and \eqref{nvtn 12})}.
 \end{eqnarray*}
\end{Small}

Noting that $\langle \sigma_1, \sigma_2, \ldots, \sigma_{n-1} \rangle \cong S_n$,  it is not difficult to see that \eqref{nvtn 10} and \eqref{nvtn 11} give
\begin{eqnarray}
\label{vtn rho eq1} && (\rho_4 \ldots \rho_{i+2})(\rho_3 \ldots \rho_{i+1})(\rho_2 \ldots \rho_i \rho_{i+1} \ldots \rho_j)(\rho_1 \ldots \rho_{i-1} \rho_i \ldots \rho_{j-1}) \\
\nonumber  &=&\rho_2 \rho_1 \rho_3 \rho_2(\rho_4 \rho_3 \rho_2 \rho_1) \ldots(\rho_{i+2} \rho_{i+1} \rho_i \rho_{i-1})(\rho_{i+3} \ldots \rho_j)(\rho_{i+2} \ldots \rho_{j-1}),\\
\nonumber & &\\
\label{vtn rho eq3} && (\rho_4 \rho_{3} \rho_2 \rho_1) \ldots (\rho_{i+2} \rho_{i+1}  \rho_i \rho_{i-1}) \\
\nonumber  &=& (\rho_4 \ldots \rho_{i+2})(\rho_3 \ldots \rho_{i+1})(\rho_2 \ldots \rho_{i})(\rho_1 \ldots \rho_{i-1}),\\
\nonumber & &\\
\label{vtn rho eq2} && (\rho_{i+1} \ldots \rho_3)(\rho_{i+2} \ldots \rho_4)(\dashuline{\rho_2} \rho_3 \ldots \rho_j)(\rho_1 \ldots \rho_{j-1})\\ 
\nonumber  &=& (\rho_2 \ldots \rho_j)(\rho_1 \ldots \rho_{j-1})(\rho_{i-1} \ldots \rho_1)(\rho_i \ldots \rho_2).
 \end{eqnarray}

Using \eqref{vtn rho eq1}, we can write
\begin{Small}
\begin{eqnarray*}
&& s_is_j\\
&=& (\rho_{j-1} \ldots \rho_{i+2})(\rho_j \ldots \rho_{i+3})(\rho_{i-1} \ldots \rho_1)(\rho_i \ldots \rho_2)(\rho_{i+1} \ldots \rho_3)(\rho_{i+2} \ldots \rho_4) \\
& & \dashuline{s_1(\rho_2 \rho_3 \rho_1 \rho_2 s_1 \rho_2 \rho_1 \rho_3 \rho_2)}(\rho_4 \rho_3 \rho_2 \rho_1) \ldots(\rho_{i+2} \rho_{i+1} \rho_i \rho_{i-1})(\rho_{i+3} \ldots \rho_j)(\rho_{i+2} \ldots \rho_{j-1}) \\
&=&(\rho_{j-1} \ldots \rho_{i+2})(\rho_j \ldots \rho_{i+3})(\rho_{i-1} \ldots \rho_1)(\rho_i \ldots \rho_2)(\rho_{i+1} \ldots \rho_3)(\dashuline{\rho_{i+2} \ldots \rho_4}) \\
& &(\dashuline{\rho_2 \rho_1}\rho_3\rho_2 s_1 \rho_2 \rho_1 \rho_3 \rho_2) s_1\dashuline{(\rho_4 \rho_3 \rho_2 \rho_1) \ldots(\rho_{i+2} \rho_{i+1} \rho_i \rho_{i-1})}(\rho_{i+3} \ldots \rho_j)(\rho_{i+2} \ldots \rho_{j-1}) ~\text{ (By \eqref{nvtn 13})} \\
& = &(\rho_{j-1} \ldots \rho_{i+2})(\rho_j \ldots \rho_{i+3})(\rho_{i-1} \ldots \rho_1)(\rho_i \ldots \rho_2)(\rho_{i+1} \ldots \rho_3 \rho_2 \rho_1)(\rho_{i+2} \ldots \rho_4 \rho_3 \rho_2)s_1 \rho_2 \rho_1\rho_3 \rho_2 s_1  \\
& & (\rho_4 \ldots \rho_{i+2})(\rho_3 \ldots \rho_{i+1})(\rho_2 \ldots \rho_i)(\rho_1 \ldots \rho_{i-1})(\dashuline{\rho_{i+3} \ldots \rho_j})(\rho_{i+2} \ldots \rho_{j-1}) ~\text{ (By \eqref{nvtn 10}, \eqref{nvtn 12} and \eqref{vtn rho eq3})}\\
& = &(\rho_{j-1} \ldots \rho_{i+2})(\rho_j \ldots \rho_{i+3})(\rho_{i-1} \ldots \rho_1)(\rho_i \ldots \rho_2)(\rho_{i+1} \ldots \rho_3 \rho_2 \rho_1)(\rho_{i+2} \ldots \rho_4 \rho_3 \rho_2)s_1 \dashuline{\rho_2 \rho_1\rho_3 \rho_2 s_1}  \\
& & (\rho_4 \ldots \rho_{i+2} \rho_{i+3} \ldots \rho_j)(\rho_3 \ldots \rho_{i+1})(\rho_2 \ldots \rho_i)(\rho_1 \ldots \rho_{i-1})(\rho_{i+2} \ldots \rho_{j-1}) ~\text{ (By \eqref{nvtn 10})}\\
& =& (\rho_{j-1} \ldots \rho_{i+2})(\rho_j \ldots \rho_{i+3})(\rho_{i-1} \ldots \rho_1)(\rho_i \ldots \rho_2)(\rho_{i+1} \ldots \rho_3 \dashuline{\rho_2 \rho_1})(\rho_{i+2} \ldots \rho_4\dashuline{ \rho_3 \rho_2)} \\
& & s_1(\rho_2 \rho_3 \rho_4 \ldots \rho_{i+2} \rho_{i+3} \ldots \rho_j)(\rho_1 \rho_2 \ldots \rho_{i+1} \rho_{i+2} \ldots \rho_{j-1}) s_1(\rho_2 \ldots \rho_i)(\rho_1 \ldots \rho_{i-1}) ~\text{ (By \eqref{nvtn 10}  and \eqref{nvtn 12})}\\
& =& (\rho_{j-1} \ldots \rho_{i+2})(\rho_j \ldots \rho_{i+3})\dashuline{(\rho_{i-1} \ldots \rho_1)(\rho_i \ldots \rho_2)(\rho_{i+1} \ldots \rho_3 )(\rho_{i+2} \ldots \rho_4)}(\rho_2 \rho_1 \rho_3 \rho_2) \\
& & s_1(\rho_2 \rho_3 \rho_4 \ldots \rho_{i+2} \rho_{i+3} \ldots \rho_j)(\rho_1 \rho_2 \ldots \rho_{i+1} \rho_{i+2} \ldots \rho_{j-1}) s_1(\rho_2 \ldots \rho_i)(\rho_1 \ldots \rho_{i-1}) ~\text{ (By \eqref{nvtn 10})}\\
& = &(\rho_{j-1} \ldots \rho_{i+2})(\rho_j \ldots \rho_{i+3})(\rho_{i-1} \rho_i \rho_{i+1} \rho_{i+2}) \ldots(\rho_1 \dashuline{\rho_2 \rho_3 \rho_4)(\rho_2} \rho_1 \rho_3 \rho_2) \\
&& s_1(\rho_2 \ldots \rho_j)(\rho_1 \ldots \rho_{j-1}) s_1(\rho_2 \ldots \rho_i)(\rho_1 \ldots \rho_{i-1})  ~\text{ (By \eqref{vtn rho eq3})}\\
&& \quad \quad \quad \vdots \quad \quad \quad \vdots \quad \quad \quad \vdots\\
&=&(\rho_{j-1} \ldots \rho_{i+2}) \rho_{i+1}(\rho_j \ldots \rho_{i+3})(\rho_{i-1} \rho_i \rho_{i+1} \rho_{i+2}) \ldots(\dashuline{\rho_1 \rho_2} \rho_3 \rho_4)(\dashuline{\rho_1} \rho_3 \rho_2) \\
& &s_1(\rho_2 \ldots \rho_j)(\rho_1 \ldots \rho_{j-1}) s_1(\rho_2 \ldots \rho_i)(\rho_1 \ldots \rho_{i-1})\\
&& \quad \quad \quad \vdots \quad \quad \quad \vdots \quad \quad \quad \vdots\\
& = & (\rho_{j-1} \ldots \rho_{i+2}) \rho_{i+1} \rho_i(\rho_j \ldots \rho_{i+3})(\rho_{i-1} \rho_i \rho_{i+1} \rho_{i+2}) \ldots(\rho_1 \rho_2 \dashuline{\rho_3 \rho_4)(\rho_3} \rho_2) \\
&& s_1(\rho_2 \ldots \rho_j)(\rho_1 \ldots \rho_{j-1}) s_1(\rho_2 \ldots \rho_i)(\rho_1 \ldots \rho_{i-1})\\
&& \quad \quad \quad \vdots \quad \quad \quad \vdots \quad \quad \quad \vdots\\
& = & (\rho_{j-1} \ldots \rho_i)(\rho_j \ldots \rho_{i+3}) \rho_{i+2}(\rho_{i-1} \rho_i \rho_{i+1} \rho_{i+2}) \ldots(\rho_1 \dashuline{\rho_2 \rho_3 \rho_4) \rho_2} \\
&& s_1(\rho_2 \ldots \rho_j)(\rho_1 \ldots \rho_{j-1}) s_1(\rho_2 \ldots \rho_i)(\rho_1 \ldots \rho_{i-1}) \\
&& \quad \quad \quad \vdots \quad \quad \quad \vdots \quad \quad \quad \vdots\\
&= & (\rho_{j-1} \ldots \rho_i)(\rho_j \ldots \rho_{i+3}) \rho_{i+2} \rho_{i+1} \dashuline{(\rho_{i-1} \rho_i \rho_{i+1} \rho_{i+2}) \ldots(\rho_1 \rho_2 \rho_3 \rho_4)}\\
&& s_1(\rho_2 \ldots \rho_j)(\rho_1 \ldots \rho_{j-1}) s_1(\rho_2 \ldots \rho_i)(\rho_1 \ldots \rho_{i-1}) \\
&= & (\rho_{j-1} \ldots \rho_i)(\rho_j \ldots \rho_{i+1}) \dashuline{(\rho_{i-1} \ldots \rho_1)}(\rho_i \ldots \rho_2)(\rho_{i+1} \ldots \rho_3)(\rho_{i+2} \ldots \rho_4) \\
& &\dashuline{s_1}(\rho_2 \ldots \rho_j)(\rho_1 \ldots \rho_{j-1}) s_1(\rho_2 \ldots \rho_i)(\rho_1 \ldots \rho_{i-1}) ~\text{ (By \eqref{vtn rho eq3})} \\
&= & (\rho_{j-1} \ldots \rho_1)(\rho_j \ldots \rho_2) s_1 \dashuline{(\rho_{i+1} \ldots \rho_3)(\rho_{i+2} \ldots \rho_4)(\rho_2 \ldots \rho_j)(\rho_1 \ldots \rho_{j-1})} s_1(\rho_2 \ldots \rho_i)(\rho_1 \ldots \rho_{i-1})\\
&& \text{ (By \eqref{nvtn 10} and \eqref{nvtn 12})}\\
&=& (\rho_{j-1} \ldots \rho_1)(\rho_j \ldots \rho_2) s_1 (\rho_2 \ldots \rho_j)(\rho_1 \ldots \rho_{j-1})(\rho_{i-1} \ldots \rho_1)(\rho_i \ldots \rho_2) s_1(\rho_2 \ldots \rho_i)(\rho_1 \ldots \rho_{i-1})
~\text{ (By \eqref{vtn rho eq2})}\\
&=& s_j s_i.
\end{eqnarray*}
\end{Small}
This completes the proof of the theorem.
\end{proof}
\medskip

Our main computational tool will be the Reidemeister-Schreier method \cite[Theorem 2.6]{MR0207802}. Given a subgroup $H$ of a group $G= \langle S \mid R\rangle$, we fix a Schreier system $\Lambda$ of coset representatives of $H$ in $G$. For an element $w \in G$, let $\overline{w}$ denote the unique coset representative of the coset of $w$ in the Schreier set $\Lambda$. Then, by  Reidemeister-Schreier Theorem, the subgroup $H$ has a presentation with generating set $$\left\{\gamma(\mu, a):=(\mu a)(\overline{\mu a})^{-1} ~\mid~ \mu \in \Lambda~\text{ and } ~a \in S \right\}$$
and set of defining relations
$$\{\tau\left(\mu r \mu^{-1}\right) ~\mid~ \mu \in \Lambda~ \text{ and } ~r \in R\}.$$
Here, $\tau$ is the rewriting process, that is, given $g=g_1 g_2 \ldots g_k \in G$ for some $g_i \in S$, we have
$$
\tau(g) =\gamma(1, g_1) \gamma(\overline{g_1}, g_2) \cdots \gamma (\overline{g_1 g_2 \ldots g_{k-1}}, g_k).
$$

We now use Theorem \ref{reduced-presentation-n4} and the Reidemeister-Schreier method to give a presentation of $VT_n^{'}$ for $n \geq 2$. Since the abelianisation of $VT_n$ is isomorphic to the elementary abelian $2$-group of order $4$,  we can take the Schreier system of coset representatives to be $$\M = \{ 1, s_1, \rho_1, s_1\rho_1 \}.$$
In view of Theorem \ref{reduced-presentation-n4}, we take $S=\{s_1, \rho_1, \rho_2,\dots, \rho_{n-1}\}$ as a generating set for $VT_n$.

\subsection*{Generators of $VT_n^{'}$:}
The generating set $\{\gamma(\mu, a) \mid  \mu\in \M, ~a\in S \}$ of  $VT_n^{'}$ can be determined explicitly. More precisely, for each $2 \le i \le n-1$, we have
\begin{small}
\begin{equation*}
\begin{aligned}
\gamma(1, s_1) & = s_1 (\overline{s_1})^{-1} = 1,\\
\gamma(1, \rho_1) & = \rho_1 (\overline{\rho_1})^{-1} =  1,\\
\gamma(1, \rho_i) & = \rho_i (\overline{\rho_i})^{-1} = \rho_i\rho_1,\\
&&\\
\gamma(s_1, s_1) & = s_1s_1 (\overline{s_1s_1})^{-1} = 1,\\
\gamma(s_1, \rho_1) & =  s_1\rho_1 (\overline{s_1\rho_1})^{-1} =  1,\\
\gamma(s_1, \rho_i) & = s_1\rho_i (\overline{s_1\rho_i})^{-1} = s_1\rho_i \rho_1 s_1,
\end{aligned}
\begin{aligned}
\gamma(\rho_1, s_1) & = \rho_1s_1 (\overline{\rho_1s_1})^{-1} = (\rho_1s_1)^2,\\
\gamma(\rho_1, \rho_1) & = \rho_1\rho_1 (\overline{\rho_1\rho_1})^{-1} = 1,\\
\gamma(\rho_1, \rho_i) & = \rho_1\rho_i (\overline{\rho_1\rho_i})^{-1} = \rho_1\rho_i,\\
& &\\
\gamma(s_1\rho_1, s_1) & = s_1\rho_1s_1 (\overline{s_1\rho_1s_1})^{-1} = (s_1\rho_1)^2,\\
\gamma(s_1\rho_1, \rho_1) & = s_1\rho_1 \rho_1 (\overline{s_1\rho_1\rho_1})^{-1} = 1,\\
\gamma(s_1\rho_1, \rho_i) & = s_1\rho_1 \rho_i (\overline{s_1\rho_1 \rho_i})^{-1}=(s_1\rho_i \rho_1 s_1)^{-1}.
\end{aligned}
\end{equation*}
\end{small}
For $2 \le i \le n-1$, setting
$$x_i := \rho_i \rho_1, \quad z_i := s_1\rho_i \rho_1 s_1\quad \textrm{and} \quad  w := (\rho_1 s_1)^2,$$
we see that $VT_n^{'}$ is generated by the set 
$$\big\{ x_i, z_i, w ~\mid~ i=2, 3, \dots, n-1 \big\}.$$

\subsection*{Relations in  $VT_n^{'}$:}
We now determine the relations in $VT_n^{'}$ corresponding to each relation in the presentation of $VT_n$ given by Theorem \ref{reduced-presentation-n4}.
\begin{enumerate}
\item First, we consider the relation $s_1^2=1$. Then we obtain
\begin{eqnarray*}
\tau(1s_1s_11) &=& \gamma(1, s_1)\gamma(\overline{s_1}, s_1)=1,\\
\tau(s_1s_1s_1s_1) &=& (\gamma(1, s_1)\gamma(s_1, s_1))^2=1,\\
\tau(\rho_1s_1s_1\rho_1) &=& \gamma(1, \rho_1)\gamma(\rho_1, s_1)\gamma(s_1\rho_1, s_1)\gamma(\rho_1, \rho_1)=1,\\
\tau(s_1\rho_1 (s_1 s_1) \rho_1s_1) &=& \gamma(1, s_1) \gamma(s_1, \rho_1)\gamma(s_1\rho_1, s_1)\gamma(\rho_1, s_1)\gamma(s_1 \rho_1, \rho_1)\gamma(s_1, s_1)=1.
\end{eqnarray*}
\item[]

\item  Next, we consider the relations $\rho_i^2=1$ for $1 \le i \le n-1$. This case gives
\begin{eqnarray*}
\tau(1\rho_i\rho_i1) &=& \gamma(1, \rho_i)\gamma(\overline{\rho_i}, \rho_i)=1,\\
\tau(s_1 \rho_i \rho_is_1) &=& \gamma(1, s_1)\gamma(s_1, \rho_i)\gamma(s_1 \rho_1, \rho_i)\gamma(s_1, s_1)=1,\\
\tau(\rho_1 \rho_i \rho_i \rho_1) &=& \gamma(1, \rho_1)\gamma(\rho_1, \rho_i)\gamma(1, \rho_i)\gamma(\rho_1, \rho_1)=1,\\
\tau(s_1\rho_1(\rho_1\rho_1)\rho_1 s_1) &=& \gamma(1, s_1) \gamma(s_1, \rho_1)\gamma(s_1\rho_1, \rho_1)\gamma(s_1, \rho_1)\gamma(s_1 \rho_1, \rho_1)\gamma(s_1, s_1)=1,\\
\tau(s_1\rho_1(\rho_i\rho_i) \rho_1 s_1) &=& \gamma(1, s_1) \gamma(s_1, \rho_1)\gamma(s_1\rho_1, \rho_i)\gamma(s_1, \rho_i)\gamma(s_1 \rho_1, \rho_1)\gamma(s_1, s_1)= z_i^{-1}z_i=1\\
&& \textrm{for}~ i\geq 2.
\end{eqnarray*}
\item[]

\item Consider the relations $(\rho_i \rho_j)^2=1$ for $|i-j| \geq 2$. In this case, we get
\begin{eqnarray*}
\tau(1(\rho_i\rho_j)^21) &=& (\gamma(1, \rho_i)\gamma(\rho_1, \rho_j))^2\\
&=&\begin{cases}
x_j^{-2}~ &  \quad \textrm{for} \quad  j \geq 3, \\
(x_ix_{j}^{-1})^2 &\quad \textrm{for} \quad  i \geq 2 \text{ and } i+2 \leq j \leq n-1,
\end{cases}\\
& &\\
\tau(s_1(\rho_i\rho_j)^2 s_1) &=& (\gamma(s_1, \rho_i)\gamma(s_1 \rho_1, \rho_j))^2\\
&=& \begin{cases}
z_j^{-2}& \quad \textrm{for}\quad  3 \leq j \leq n-1,\\
(z_iz_{j}^{-1})^2 & \quad \textrm{for} \quad 2 \leq i \leq n-2  \quad \text{and} \quad j \geq i+2,
\end{cases}\\
& &\\
\tau(\rho_1 (\rho_i\rho_j)^2 \rho_1) &=&(\gamma(\rho_1, \rho_i)\gamma(1, \rho_j))^2\\
&=&\begin{cases}
x_j^{2} & \quad \textrm{for}\quad  3 \leq j \leq n-1, \\
(x_i^{-1}x_{j})^2 & \quad \textrm{for} \quad  i \geq 2 \quad \text{and} \quad i+2 \leq j \leq n-1,
\end{cases}\\
& &\\
\tau(s_1 \rho_1(\rho_i\rho_j)^2 \rho_1 s_1) &=& (\gamma(s_1 \rho_1, \rho_i)\gamma(s_1, \rho_j))^2 \\
&=&\begin{cases}
z_j^{2} & \quad \textrm{for} \quad j \geq 3, \\
(z_i^{-1}z_{j})^2 & \quad \textrm{for}\quad i \geq 2  \quad \text{and} \quad i+2 \leq j \leq n-1.
\end{cases}
\end{eqnarray*}

Thus, the non-trivial relations  in $VT_n^{'}$ are
\begin{eqnarray}
\label{vtn comm rel 5} x_j^2&= 1&  \quad \textrm{for} \quad 3\leq j\leq n-1,\\
\label{vtn comm rel 6} z_j^{2}&=1&   \quad \textrm{for}\quad  3\leq j\leq n-1,\\
\label{vtn comm rel 7} (x_ix_{j}^{-1})^2&= 1&  \quad \textrm{for} \quad 2 \leq i \leq n-2  \quad \text{and} \quad j \geq i+2,\\
\label{vtn comm rel 8} (z_iz_{j}^{-1})^2&=1&   \quad \textrm{for} \quad 2 \leq i \leq n-2  \quad \text{and} \quad j \geq i+2.
\end{eqnarray}

\item Considering the relations $(\rho_i \rho_{i+1})^3=1$ for $1 \le i \le n-2$ give
\begin{eqnarray*}
\tau(1(\rho_i\rho_{i+1})^3 1) &=& \big ( \gamma(1, \rho_i)\gamma(\rho_1, \rho_{i+1})\big ) ^3\\
&=&\begin{cases}
x_2^{-3} & \quad \textrm{for} \quad i=1,\\
(x_ix_{i+1}^{-1})^3 & \quad \textrm{for} \quad 2\leq i\leq n-2,
\end{cases}\\
& &\\
\tau(s_1(\rho_i\rho_{i+1})^3s_1)&=&\big ( \gamma(s_1, \rho_i)\gamma(s_1 \rho_1, \rho_{i+1})\big )^3\\
&=&\begin{cases}
z_2^{-3} & \quad \textrm{for} \quad i=1,\\
(z_iz_{i+1}^{-1})^3& \quad \textrm{for} \quad 2 \leq i \leq n-2,
\end{cases}\\
& &\\
\tau(\rho_1(\rho_i\rho_{i+1})^3 \rho_1)&=& \big ( \gamma(\rho_1, \rho_i)\gamma(1, \rho_{i+1})\big ) ^3\\
&=&\begin{cases}
x_2^{3}& \quad \textrm{for} \quad i=1,\\
(x_i^{-1}x_{i+1})^3 & \quad \textrm{for} \quad  2 \leq i \leq n-2,
\end{cases}\\
& &\\
\tau(s_1 \rho_1(\rho_i\rho_{i+1})^3 \rho_1 s_1) &=& \big ( \gamma(s_1 \rho_1, \rho_i)\gamma(s_1, \rho_{i+1})\big ) ^3\\
&=&\begin{cases}
z_2^{3}&\quad \textrm{for} \quad i=1,\\
(z_i^{-1}z_{i+1})^3 & \quad \textrm{for} \quad 2\leq i\leq n-2.
\end{cases}
\end{eqnarray*}
Thus, we have the following non-trivial relations  in $VT_n^{'}$
\begin{eqnarray}
\label{vtn comm rel 1} x_2^3&=1,&\\
\label{vtn comm rel 2} z_2^{3}&=1,&\\
\label{vtn comm rel 3} (x_ix_{i+1}^{-1})^3&=1& \quad \textrm{for} \quad 2\leq i\leq n-2,\\
\label{vtn comm rel 4} (z_iz_{i+1}^{-1})^3&=1& \quad \textrm{for}\quad 2\leq i\leq n-2.
\end{eqnarray}

\item Considering the relations $(\rho_i s_1)^2=1$ for $ 3 \leq i \leq n-1$ give
\begin{eqnarray*}
\tau(1(\rho_i s_1)^2 1) &=& \gamma(1, \rho_i)\gamma(\rho_1, s_1) \gamma(s_1 \rho_1, \rho_i) = x_i w z_i^{-1} \quad \textrm{for}\quad i \geq 3,\\
&&\\
\tau(s_1(\rho_i s_1)^2 s_1) &=& \gamma(s_1, \rho_i)\gamma(s_1\rho_1, s_1) \gamma(\rho_1, \rho_i) = z_i w^{-1} x_i^{-1} \quad \textrm{for} \quad i \geq 3,\\
&&\\
\tau(\rho_1 (\rho_i s_1)^2 \rho_1) &=& \gamma(\rho_1, \rho_i) \gamma(s_1, \rho_i) \gamma(s_1\rho_1, s_1) = x_i^{-1}z_i w^{-1} \quad \textrm{for} \quad i \geq 3,\\
&&\\
\tau( s_1 \rho_1 (\rho_i s_1)^2 \rho_1 s_1) &=&  \gamma(s_1 \rho_1, \rho_i) \gamma(1, \rho_i) \gamma( \rho_1, s_1) = z_i^{-1} x_i w \quad \textrm{for} \quad i \geq 3.
\end{eqnarray*}

Thus, we have the following non-trivial relations in  $VT_n^{'}$
\begin{eqnarray}\label{zi in terms of xi}
\label{vtn comm rel 9} z_i &=& x_i w   \quad \textrm{~for} \quad  3 \leq i \leq n-1.
\end{eqnarray}
\item[]

\item Finally, we consider the relation $(s_1\rho_2 \rho_1\rho_3 \rho_2)^4 = 1$. Then we obtain
\begin{eqnarray*}
&& \tau(1(s_1\rho_2\rho_3\rho_1\rho_2s_1\rho_2\rho_1\rho_3\rho_2)^2 1)\\
&=& (\gamma(s_1,\rho_2)\gamma(s_1 \rho_1, \rho_3)\gamma(s_1 \rho_1, \rho_2)\gamma(1, \rho_2)\gamma(1, \rho_3)\gamma(\rho_1,\rho_2))^2\\
&=& (z_2z_3^{-1}z_2^{-1}x_2x_3x_2^{-1})^2,\\
&&\\
&& \tau(s_1(s_1\rho_2\rho_3\rho_1\rho_2s_1\rho_2\rho_1\rho_3\rho_2)^2 s_1)\\
&=& (\gamma(1,\rho_2)\gamma(\rho_1, \rho_3)\gamma(\rho_1, \rho_2)\gamma(s_1, \rho_2)\gamma(s_1, \rho_3)\gamma(s_1\rho_1,\rho_2))^2\\
&=& (x_2x_3^{-1}x_2^{-1}z_2z_3z_2^{-1})^2,\\
&& \\
&& \tau(\rho_1(s_1\rho_2\rho_3\rho_1\rho_2s_1\rho_2\rho_1\rho_3\rho_2)^2 \rho_1)\\
&=& (\gamma(\rho_1, s_1)\gamma(s_1\rho_1,\rho_2)\gamma(s_1, \rho_3)\gamma(s_1, \rho_2)\gamma(s_1\rho_1, s_1)\gamma(\rho_1, \rho_2)\gamma(\rho_1, \rho_3)\gamma(1,\rho_2))^2\\
&=& (w z_2^{-1}z_3z_2 w^{-1}x_2^{-1}x_3^{-1}x_2)^2,\\
&& \\
&&\tau(s_1 \rho_1 (s_1\rho_2\rho_3\rho_1\rho_2s_1\rho_2\rho_1\rho_3\rho_2)^2 \rho_1 s_1)\\
&=&   (\gamma(s_1 \rho_1, s_1)\gamma(\rho_1,\rho_2)\gamma(1, \rho_3)\gamma(1, \rho_2)\gamma(\rho_1, s_1) \gamma(s_1\rho_1, \rho_2)\gamma(s_1 \rho_1, \rho_3)\gamma(s_1,\rho_2))^2\\
&=& (w^{-1}x_2^{-1}x_3x_2 w z_2^{-1}z_3^{-1}z_2)^2.
\end{eqnarray*}

Thus, we get the following two non-trivial relations in  $VT_n^{'}$
\begin{eqnarray}
\label{vtn comm rel 10}  (z_2z_3^{-1}z_2^{-1}x_2x_3x_2^{-1})^2 &= &1,\\
\label{vtn comm rel 11} (w z_2^{-1}z_3z_2 w^{-1}x_2^{-1}x_3^{-1}x_2)^2 &=& 1.
\end{eqnarray}
\end{enumerate}

We can eliminate the generators $z_i$ for $3 \leq i \leq n-1$ using \eqref{zi in terms of xi}. Setting $z:=z_2$, we see that the relations \eqref{vtn comm rel 1} through \eqref{vtn comm rel 11}  yield the following result.

\begin{theorem}\label{commutator-subgroup-vtn}
The following hold for the commutator subgroup $VT_n^{'}$ of $VT_n$:
\begin{enumerate}
\item $VT_2^{'} \cong \mathbb{Z}$ and is generated by $(\rho_1 s_1)^2$.

\item $VT_3^{'} \cong \mathbb{Z}_3 * \mathbb{Z}_3 * \mathbb{Z}$ and has a presentation
$$ \big\langle \rho_2\rho_1,  s_1\rho_2\rho_1s_1, (\rho_1 s_1)^2 ~\mid~ (\rho_2\rho_1)^3 = (s_1\rho_2\rho_1s_1)^3=1 \big\rangle.$$

\item For  $n \geq 4$, $VT_n^{'}$ has a presentation with generating set
$$\big\{ x_i, z, w~\mid~ i=2, 3, \dots, n-1 \big\}$$
and  defining relations  as follows:
\begin{eqnarray*}
x_2^3&=&1,\\
x_j^2&=&1 \quad \textrm{for} \quad 3\leq j\leq n-1,\\
z^{3}&=&1,\\
(x_ix_{i+1}^{-1})^3&=&1 \quad \textrm{for} \quad 2\leq i\leq n-2,\\
(x_ix_{j}^{-1})^2&=&1  \quad \textrm{for} \quad 2 \leq i \leq n-2 \quad \text{and} \quad j \geq i+2,\\
(x_j w)^{2}&=&1  \quad \textrm{for} \quad 3\leq j\leq n-1,\\
(z w^{-1}x_3^{-1})^3&=&1,\\
(z w^{-1}x_j^{-1})^2&=&1  \quad \textrm{for} \quad 4 \leq j\leq n-1,\\
(z w^{-1}x_3^{-1}z^{-1}x_2x_3x_2^{-1})^2&=&1,\\
(w z^{-1}x_3 w z w^{-1}x_2^{-1}x_3^{-1}x_2)^2&=&1.
\end{eqnarray*}
\end{enumerate}
\end{theorem}

It is known that the second and the third term of the lower central series coincide for the braid group $B_n$ when $n \geq 3$ \cite{MR0251712} and coincide for the virtual braid group $VB_n$ when $n \geq 4$ \cite[Proposition 7(c)]{MR2493369}. The same assertion does not hold for twin groups \cite[Theorem 4.5]{MR4017601}. However, it turns out that the assertion holds for virtual twin groups.

\begin{proposition}\label{stable-lower-central-series-tn}
 $VT_n^{'} = \gamma_3 (VT_n)$ for $n \geq 3$.
\end{proposition}
 
\begin{proof}
It is enough to prove that $VT_n / \gamma_3 (VT_n)$ is abelian. The group $VT_n / \gamma_3 (VT_n)$ is generated by 
$$\tilde{s}_i := s_i \gamma_3 (VT_n) \quad \textrm{and} \quad \tilde{\rho}_i := \rho_i \gamma_3 (VT_n),$$ where $1 \le i \le n-1$. It is easy to see that 
$$\rho_{i+1} = \rho_i [ \rho_i, [\rho_i, \rho_{i+1}]] \quad \textrm{and} \quad s_{i+1} = s_i {[[\rho_{i+1}, [\rho_{i+1}, \rho_i]], s_i]}^{-1}.$$
This gives 
$$\tilde{\rho}_{i+1}= \rho_{i+1} \gamma_3 (VT_n)= \rho_i [ \rho_i, [\rho_i, \rho_{i+1}]] \gamma_3 (VT_n)= \rho_i \gamma_3 (VT_n) = \tilde{\rho}_i$$
and 
$$\tilde{s}_{i+1} = s_{i+1} \gamma_3 (VT_n)= s_i {[[\rho_{i+1}, [\rho_{i+1}, \rho_i]], s_i]}^{-1}\gamma_3 (VT_n) =s_i \gamma_3 (VT_n) = \tilde{s}_i. $$
for each $1 \le i \le n-2$. Since
$$\tilde{\rho}_{1} \tilde{s}_{1} = \tilde{\rho}_{3} \tilde{s}_{1} = \tilde{s}_{1} \tilde{\rho}_{3} = \tilde{s}_{1} \tilde{\rho}_{1},$$
the assertion follows.
\end{proof}

\begin{corollary}\label{vtn residually nilpotent}
$VT_n$ is residually nilpotent if and only if $n =2$. 
\end{corollary}

We conclude the section with a result on freeness of commutator subgroup of $PVT_n$. A graph is called {\it chordal}  if each of its cycles with more than three  vertices has a chord (an edge joining two vertices that are not adjacent in the cycle).  A {\it clique} (or a complete subgraph) of a graph is a subset $C$ of vertices such that every two vertices in $C$ are connected by an edge. It is well-known that a graph is chordal if and only if its vertices can be ordered in such a way that the lesser neighbours of each vertex form a clique.

\begin{proposition}\label{commutator-subgroup-freeness}
The commutator subgroup of $PVT_n$ is free if and only if $n \le 4$.
\end{proposition}

\begin{proof}
By \cite[Theorem 3.3]{NaikNandaSingh2}, the pure virtual twin group  $PVT_n$ is an irreducible right-angled Artin group for each $n \ge 2$, and has a presentation
$$PVT_n= \big\langle \lambda_{i,j},~1 \leq i < j \leq n ~\mid~ \lambda_{i,j} \lambda_{k,l} =  \lambda_{k,l} \lambda_{i,j} \text{ for distinct integers } i, j, k, l \big\rangle,$$
where $\lambda_{i, i+1}= s_i \rho_i$ for each $1 \le i \le n-1$ and $\lambda_{i,j} = \rho_{j-1} \rho_{j-2} \dots \rho_{i+1} \lambda_{i, i+1} \rho_{i+1} \dots \rho_{j-2}  \rho_{j-1}$
for each $1 \leq i < j \leq n$ and $j \ne i+1$.  The assertion is now immediate for $n=2, 3$. 
\par
The graph of $PVT_4$ (see Figure \ref{graph-pvt4}) is vacuously chordal. 

\begin{figure}
\centering
\begin{minipage}{.50\textwidth}
\centering
\includegraphics[width=.6\linewidth]{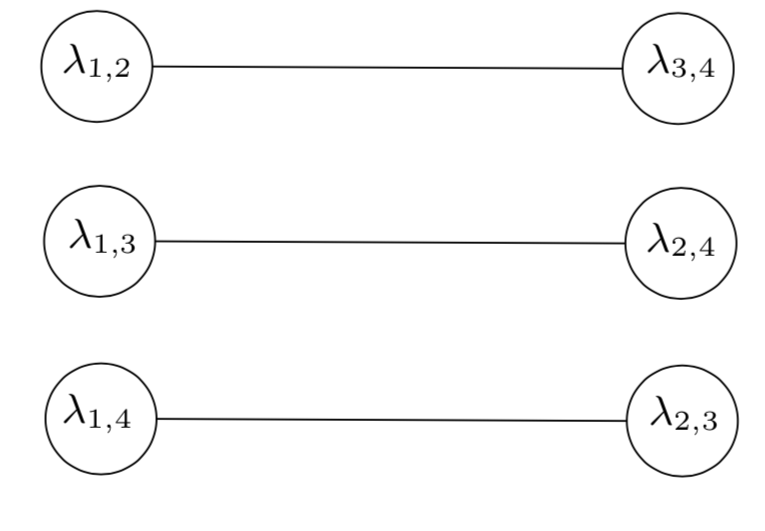}
\captionof{figure}{The graph of $PVT_4$.}
\label{graph-pvt4}
\end{minipage}
\begin{minipage}{.45\textwidth}
\centering
\includegraphics[width=.8\linewidth]{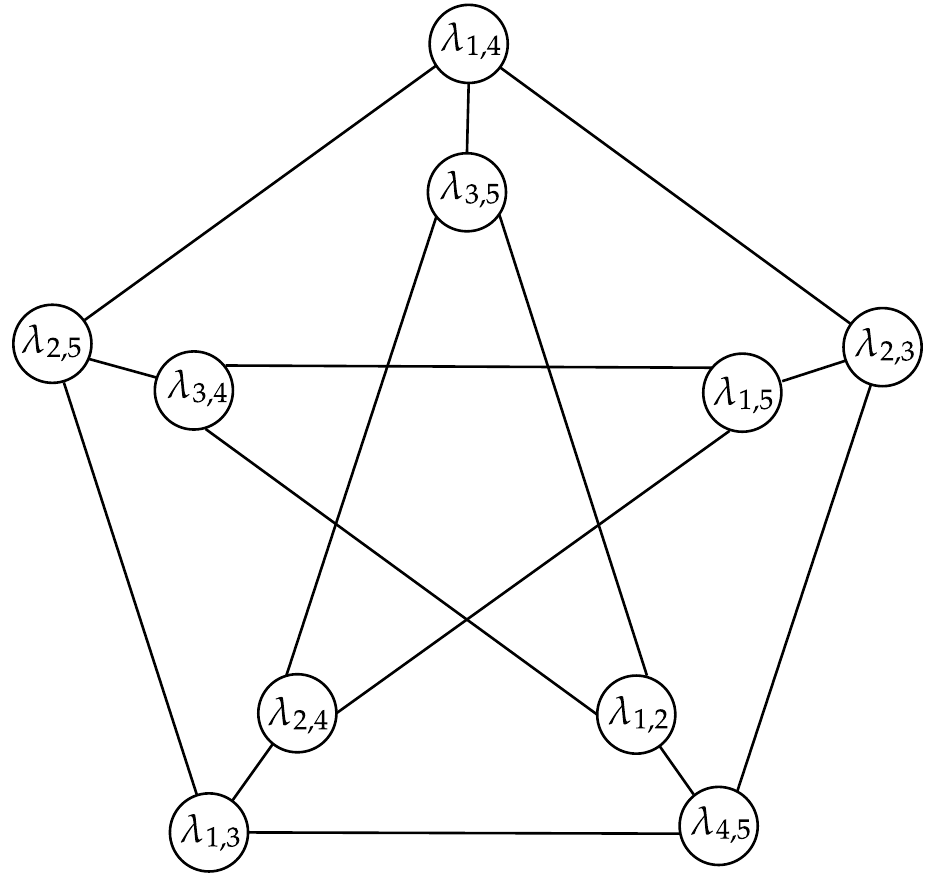}
\captionof{figure}{The graph of $PVT_5$.}
\label{graph-pvt5}
\end{minipage}
\end{figure}

By \cite[Corollary 4.4]{MR3588981}, the commutator subgroup of a right-angled Artin group is free if and only if its associated graph is chordal, and hence $PVT_4$ is free.
\par
For $n\geq 5$, fix an ordering on the vertex set of the graph, for example, it could be the lexicographic ordering in our case. Let $\lambda_{i, j}$ be a maximal vertex and $p, q, r \in \{1, 2, \ldots, n\}\setminus \{i, j\}$ be three distinct integers. Then both $\lambda_{p, q}$ and $\lambda_{p, r}$ are lesser neighbours of $\lambda_{i, j}$, but there cannot be an edge between $\lambda_{p, q}$ and $\lambda_{p, r}$. Thus, lesser neighbours of the vertex $\lambda_{i, j}$ do not form a clique, and hence the graph is not chordal.  
\end{proof}
\medskip

\section{Presentation of commutator subgroup of virtual triplet group}\label{sec commutator subgroup of virtual triplet}
In this section, we determine a finite presentation of the commutator subgroup $VL_n^{'}$ of $VL_n$. The approach is similar to that of the preceding section. We will need the following reduced presentation of $VL_n$.

\begin{theorem}\label{reduced-presentation-n4-Ln}
The virtual triplet group $VL_n$ ($n\geq 3$) has a presentation with generating set $\{y_1, \rho_1, \rho_2,\dots,\rho_{n-1}\}$ and defining relations as follows:
\begin{align} 
y_1^{2} &=1, \label{nvln 21} \\
\rho_i^{2} &= 1 \quad \textrm{for} \quad 1 \le i \le n-1,\label{nvln 22} \\
\rho_i\rho_j &= \rho_j\rho_i \quad \textrm{for} \quad |i - j| \geq 2,\label{nvln 23} \\
\rho_i\rho_{i+1}\rho_i &= \rho_{i+1}\rho_i\rho_{i+1} \quad \textrm{for} \quad 1 \le i \le n-2,\label{nvln 24} \\
\rho_iy_1 &= y_1\rho_i \quad \textrm{for} \quad i \geq 3,\label{nvln 25} \\
(y_1\rho_1 \rho_2 y_1 \rho_2 \rho_1)^3 &= 1. \label{nvln 26}
\end{align}
\end{theorem}

\begin{proof}
Replacing  $s_i$ by $y_i$ in Claim 1 in the proof of Theorem \ref{reduced-presentation-n4} gives
\begin{equation}\label{yi in terms of y1}
y_{i+1}=(\rho_i \rho_{i-1}\dots \rho_2\rho_1)(\rho_{i+1} \rho_{i}\dots \rho_3\rho_2)y_1(\rho_{2} \rho_{3}\dots \rho_i\rho_{i+1})(\rho_1 \rho_2\dots \rho_{i-1}\rho_i)
\end{equation}
 for $i \geq 2$. Hence, we can eliminate the generators $y_i$ for $i \geq 2$.
\par

Next, we show that the relations of $VL_n$ involving $y_i$ for $i \geq 2$ can be recovered from the relations listed in the statement of the theorem. 
Firstly, replacing  $s_i$ by $y_i$ in Claim 2 in the proof of Theorem \ref{reduced-presentation-n4}, we see that the relations $\rho_i y_{i+1}\rho_i=\rho_{i+1} y_{i}\rho_{i+1}$ for $1 \le i \le n-2$ can be recovered from the relations \eqref{nvln 22} and \eqref{nvln 23}. Similarly, the relation $y_i \rho_j = \rho_j y_i$ for $|i-j| \geq 2$ is a consequence of the relations \ref{nvln 23}, \ref{nvln 24} and \ref{nvln 25}. 
\par

Next, we claim that the relation $y_i y_{i+1} y_i = y_{i+1} y_i y_{i+1}$ is a consequence of \eqref{yi in terms of y1} and the relations \eqref{nvln 21} through \eqref{nvln 26}. We compute
\begin{Small}
\begin{eqnarray*}
&&  y_i y_{i+1} y_i\\
 &=& (\rho_{i-1}\ldots \rho_{1})(\rho_{i}\ldots\rho_{2})y_{1}(\rho_{2}\ldots\rho_{i})\dashuline{(\rho_1 \ldots \rho_{i-1})} \dashuline{(\rho_{i}\ldots \rho_{1})}(\rho_{i+1}\ldots\rho_{2})y_{1}\\
& & (\rho_{2}\ldots\rho_{i+1})\dashuline{(\rho_1 \ldots \rho_{i})}\dashuline{(\rho_{i-1}\ldots \rho_{1})}(\rho_{i}\ldots\rho_{2})y_{1}(\rho_{2}\ldots\rho_{i}) (\rho_1 \ldots \rho_{i-1}) ~\text{ (By \eqref{yi in terms of y1})}\\
 &=& \left(\rho_{i-1} \ldots \rho_1\right)\left(\rho_i \ldots \rho_2\right) y_1( \dashuline{\left.\rho_2 \ldots \rho_i\right)\left(\rho_i \ldots \rho_2\right.} \rho_1 \rho_2 \ldots \rho_i)\left(\rho_{i+1} \ldots \rho_2\right) y_1 \\
& & \left(\rho_2 \ldots \rho_{i+1}\right) (\rho_i \ldots \rho_2 \rho_1 \dashuline{\left.\rho_2 \ldots \rho_i ) (\rho_i \ldots \rho_2\right.} ) y_1\left(\rho_2 \ldots \rho_i\right)\left(\rho_1 \ldots \rho_{i-1}\right)  ~\text{ (By \eqref{nvln 24})}\\
 &=& \left(\rho_{i-1} \ldots \rho_1\right)\left(\rho_i \ldots \rho_2\right) y_1 (\rho_1 \dashuline{\left.\rho_2 \ldots \rho_i ) (\rho_{i+1} \rho_i \ldots \rho_2\right.}) y_1(\dashuline{\left.\rho_2 \ldots \rho_i \rho_{i+1}) (\rho_i \ldots \rho_2\right.} \rho_1) y_1 \\
& & \left(\rho_2 \ldots \rho_i\right)\left(\rho_1 \ldots \rho_{i-1}\right) ~\text{ (By \eqref{nvln 22})}\\
 &=& \left(\rho_{i-1} \ldots \rho_1\right)\left(\rho_i \ldots \rho_2\right) \dashuline{y_1 \rho_1} (\dashuline{\rho_{i+1} \ldots \rho_3} \rho_2 \dashuline{\rho_3 \ldots \rho_{i+1}}) \dashuline{y_1} (\rho_{i+1} \ldots \rho_3 \rho_2 \dashuline{\rho_3 \ldots \rho_{i+1}} ) \dashuline{\rho_1 y_1}\\
& & \left(\rho_2 \ldots \rho_i\right)\left(\rho_1 \ldots \rho_{i-1}\right)  ~\text{ (By \eqref{nvln 24})}\\
 &=& \left(\rho_{i-1} \ldots \rho_1\right)\left(\rho_i \ldots \rho_2\right)\left(\rho_{i+1} \ldots \rho_3\right) y_1 \rho_1 \rho_2 y_1\dashuline{\left(\rho_3 \ldots \rho_{i+1}\right)\left(\rho_{i+1} \ldots \rho_3\right)} \rho_2 \rho_1 y_1\left(\rho_3 \ldots \rho_{i+1}\right) \\
& & \left(\rho_2 \ldots \rho_i\right)\left(\rho_1 \ldots \rho_{i-1}\right)  ~\text{ (By \eqref{nvln 23} and \eqref{nvln 25})}\\
 &=& \left(\rho_{i-1} \ldots \rho_1\right)\left(\rho_i \ldots \rho_2\right)\left(\rho_{i+1} \ldots \rho_3\right) \dashuline{y_1 \rho_1 \rho_2 y_1 \rho_2 \rho_1} y_1\left(\rho_3 \ldots \rho_{i+1}\right)\left(\rho_2 \ldots \rho_i\right)\left(\rho_1 \ldots \rho_{i-1}\right)~\text{ (By \eqref{nvln 22})}\\
 &=& \left(\rho_{i-1} \ldots \rho_1\right)\left(\rho_i \ldots \rho_2\right)(\dashuline{\rho_{i+1} \ldots \rho_3} ) \dashuline{\rho_1} \rho_2 y_1 \rho_2 \rho_1 y_1 \rho_1 \rho_2 y_1 \rho_2 \dashuline{\rho_1}(\dashuline{\rho_3 \ldots \rho_{i+1}})\left(\rho_2 \ldots \rho_i\right) \left(\rho_1 \ldots \rho_{i-1}\right)\\
& & \text{(By \eqref{nvln 26})}\\
&=&  (\rho_{i-1} \ldots \rho_1)(\rho_i \ldots \rho_1)(\rho_{i+1} \ldots \rho_2) y_1 \rho_2 \rho_1 y_1 \rho_1 \rho_2 y_1 (\rho_2 \ldots \rho_{i+1}) (\rho_1 \ldots \rho_i)(\rho_1 \ldots \rho_{i-1})~\text{(By \eqref{nvln 23}})
\end{eqnarray*}
\end{Small}
and
\begin{Small}
\begin{eqnarray*}
&& y_{i+1} y_{i} y_{i+1}\\ 
 &=&(\rho_{i}\ldots \rho_{1})(\rho_{i+1}\ldots\rho_{2})y_{1}(\rho_{2}\ldots\rho_{i+1})\dashuline{(\rho_1 \ldots \rho_{i})} \dashuline{(\rho_{i-1}\ldots \rho_{1})}(\rho_{i}\ldots\rho_{2})y_{1}\\
& & (\rho_{2}\ldots\rho_{i})\dashuline{(\rho_1 \ldots \rho_{i-1})}\dashuline{(\rho_{i}\ldots \rho_{1})}(\rho_{i+1}\ldots\rho_{2})y_{1}(\rho_{2}\ldots\rho_{i+1})(\rho_1 \ldots \rho_{i})  ~\text{ (By \eqref{yi in terms of y1})}\\
&=&\left(\rho_{i} \ldots \rho_1\right)\left(\rho_{i+1} \ldots \rho_2\right) y_1\left(\rho_2 \ldots \rho_{i+1}\right)\left(\rho_i \ldots \rho_2 \rho_1\right.\dashuline{\left. \rho_2 \ldots \rho_i\right)\left(\rho_{i} \ldots \rho_2\right)} y_1 \\
& & \dashuline{\left(\rho_2 \ldots \rho_{i}\right)\left(\rho_i \ldots \rho_2\right.}\left. \rho_1 \rho_2 \ldots \rho_i\right)\left(\rho_{i+1} \ldots \rho_2\right) y_1\left(\rho_2 \ldots \rho_{i+1}\right)\left(\rho_1 \ldots \rho_{i}\right)  ~\text{ (By \eqref{nvln 24})}\\
 &=& \left(\rho_{i} \ldots \rho_1\right)\left(\rho_{i+1} \ldots \rho_2\right) y_1\dashuline{\left(\rho_2 \ldots \rho_{i+1}\right)\left(\rho_i \ldots \rho_2\right.}\left. \rho_1\right) y_1 \left(\rho_1 \right.\dashuline{\left.\rho_2 \ldots \rho_i\right)\left(\rho_{i+1} \ldots \rho_2\right)} \\
& & y_1\left(\rho_2 \ldots \rho_{i+1}\right)\left(\rho_1 \ldots \rho_{i}\right)  ~\text{ (By \eqref{nvln 22})}\\
&=& \left(\rho_{i} \ldots \rho_1\right)\left(\rho_{i+1} \ldots \rho_2\right) \dashuline{y_1} (\dashuline{\rho_{i+1} \ldots \rho_3} \rho_2 \dashuline{\rho_3 \ldots \rho_{i+1}})\dashuline{\rho_1 y_1\rho_1}(\rho_{i+1} \ldots \rho_3 \rho_2 \dashuline{\rho_3 \ldots \rho_{i+1}})\dashuline{y_1} \\
& & \left(\rho_2 \ldots \rho_{i+1}\right)\left(\rho_1 \ldots \rho_{i}\right) ~\text{ (By \eqref{nvln 24})}\\
 &=& \left(\rho_{i} \ldots \rho_1\right)\left(\rho_{i+1} \ldots \rho_2\right) \left(\rho_{i+1} \ldots \rho_3) \right.\dashuline{y_1 \rho_2 \rho_1 y_1\rho_1\rho_2 y_1}\left. (\rho_3 \ldots \rho_{i+1}\right) \left(\rho_2 \ldots \rho_{i+1}\right)\left(\rho_1 \ldots \rho_{i}\right) \\
& & \text{(By \eqref{nvln 22}, \eqref{nvln 23} and \eqref{nvln 25})}.
\end{eqnarray*}
\end{Small}
But, \eqref{nvln 23} and  \eqref{nvln 24} give
$$\left(\rho_i \ldots \rho_1\right)\left(\rho_{i+1} \rho_i \rho_{i-1} \ldots \rho_2\right)(\rho_{i+1}\rho_i \ldots \rho_3 )=
\left(\rho_{i-1} \ldots \rho_1\right)\left(\rho_i \ldots \rho_1\right)\left(\rho_{i+1} \ldots \rho_2\right),$$
and the claim holds, which completes the proof.
\end{proof}
\medskip

Since the abelianisation of $VL_n$ is isomorphic to the elementary abelian $2$-group of order $4$,  we can take $$\M = \{ 1, y_1, \rho_1, y_1\rho_1 \}$$ as a Schreier system of coset representatives. We use the presentation of $VL_n$ given by Theorem \ref{reduced-presentation-n4-Ln} and set $S=\{y_1, \rho_1, \rho_2,\dots, \rho_{n-1}\}$.

\subsection*{Generators of $VL_n^{'}$:}
We first compute the generating set $\{\gamma(\mu, a) \mid \mu\in \M, ~a\in S \}$ explicitly. Since the generating set and the Schreier system is similar to the case of $VT_n$, ignoring the trivial generators, we see that $VL_n^{'}$ is generated by the set
$$\big\{ \alpha_i, \beta_i, \delta ~\mid ~ i=2, 3, \dots, n-1 \big\},$$
where
\begin{eqnarray*}
\alpha_i &:=& \rho_i \rho_1 = \gamma(1, \rho_i) = \gamma(\rho_1, \rho_i)^{-1},\\
 \beta_i &:=& y_1\rho_i \rho_1 y_1 = \gamma(y_1, \rho_i) =\gamma(y_1\rho_1, \rho_i)^{-1},\\
 \delta &:=& (\rho_1 y_1)^2 =  \gamma(\rho_1, y_1)=\gamma(y_1\rho_1, y_1)^{-1}.
\end{eqnarray*}

\subsection*{Relations in  $VL_n^{'}$:}
We now determine the relations in $VL_n^{'}$ corresponding to each relation in the presentation of $VL_n$ given by Theorem \ref{reduced-presentation-n4-Ln}. Since the computations are similar to the ones in the proof of Theorem \ref{commutator-subgroup-vtn}, we elaborate only those which yield non-trivial relations in $VL_n^{'}$.

\begin{enumerate}
\item The relation $y_1^2=1$ gives only trivial relations in $VL_n^{'}$.
\item Similarly, the relations $\rho_i^2=1$ for $1 \le i \le n-1$ give only trivial relations in $VL_n^{'}$.
\item Next, we consider the relations $(\rho_i \rho_j)^2=1$ for $|i-j| \geq 2$. In this case, we get
\begin{eqnarray*}
\tau(1(\rho_i\rho_j)^21) &=& (\gamma(1, \rho_i)\gamma(\rho_1, \rho_j))^2\\
&=&\begin{cases}
\alpha_j^{-2}~ &  \quad \textrm{for} \quad  j \geq 3, \\
(\alpha_i \alpha_{j}^{-1})^2 & \quad \textrm{for} \quad  i \geq 2  \quad \text{and}  \quad i+2 \leq j \leq n-1,
\end{cases}\\
& &\\
\tau(y_1(\rho_i\rho_j)^2 y_1) &=&(\gamma(y_1, \rho_i)\gamma(y_1 \rho_1, \rho_j))^2\\
&=& \begin{cases}
\beta_j^{-2}& \quad \textrm{for} \quad  3 \leq j \leq n-1,\\
(\beta_i \beta_{j}^{-1})^2 & \quad \textrm{for}   \quad  2 \leq i \leq n-2  \quad  \text{and} \quad  j \geq i+2,
\end{cases}\\
& &\\
\tau(\rho_1 (\rho_i\rho_j)^2 \rho_1) &=&(\gamma(\rho_1, \rho_i)\gamma(1, \rho_j))^2\\
&=&\begin{cases}
\alpha_j^{2} & \quad \textrm{for} \quad  3 \leq j \leq n-1, \\
(\alpha_i^{-1} \alpha_{j})^2 &  \quad \textrm{for} \quad  i \geq 2  \quad \text{and} \quad  i+2 \leq j \leq n-1,
\end{cases}\\
& &\\
\tau(y_1 \rho_1(\rho_i\rho_j)^2 \rho_1 y_1)&=& (\gamma(y_1 \rho_1, \rho_i)\gamma(y_1, \rho_j))^2\\
&=&\begin{cases}
\beta_j^{2} &  \quad \textrm{for} \quad  j \geq 3, \\
(\beta_i^{-1} \beta_{j})^2 &  \quad \textrm{for} \quad  i \geq 2  \quad \text{and} \quad  i+2 \leq j \leq n-1.
\end{cases}
\end{eqnarray*}

Thus, we have the following non-trivial relations in $VL_n^{'}$
\begin{eqnarray}
\label{vln comm rel 5} \alpha_j^2&= 1&   \quad \textrm{for} \quad 3\leq j\leq n-1,\\
\label{vln comm rel 6} \beta_j^{2}&=1&    \quad \textrm{for} \quad  3\leq j\leq n-1,\\
\label{vln comm rel 7} (\alpha_i \alpha_{j}^{-1})^2&= 1&   \quad \textrm{for} \quad  2 \leq i \leq n-2   \quad \text{and} \quad  j \geq i+2,\\
\label{vln comm rel 8} (\beta_i \beta_{j}^{-1})^2&=1&   \quad \textrm{for} \quad  2 \leq i \leq n-2   \quad \text{and} \quad  j \geq i+2.
\end{eqnarray}
\item[]

\item We now consider the relations $(\rho_i \rho_{i+1})^3=1$ for $1 \le i \le n-2$. Then we obtain
\begin{eqnarray*}
\tau(1(\rho_i\rho_{i+1})^3 1) &=& \big ( \gamma(1, \rho_i)\gamma(\rho_1, \rho_{i+1})\big ) ^3\\
&=&\begin{cases}
\alpha_2^{-3} & \quad \textrm{for} \quad i=1,\\
(\alpha_i\alpha_{i+1}^{-1})^3 & \quad \textrm{for} \quad 2\leq i\leq n-2,
\end{cases}\\
&&\\
\tau(y_1(\rho_i\rho_{i+1})^3y_1)&=& \big ( \gamma(y_1, \rho_i)\gamma(y_1\rho_1, \rho_{i+1})\big ) ^3\\
&=&\begin{cases}
\beta_2^{-3} & \quad \textrm{for}  \quad i=1,\\
(\beta_i \beta_{i+1}^{-1})^3& \quad \textrm{for} \quad  2 \leq i \leq n-2,
\end{cases}\\
&&\\
\tau(\rho_1(\rho_i\rho_{i+1})^3 \rho_1)&=& \big ( \gamma(\rho_1, \rho_i)\gamma(1, \rho_{i+1})\big ) ^3\\
&=&\begin{cases}
\alpha_2^{3}& \quad \textrm{for} \quad i=1,\\
(\alpha_i^{-1}\alpha_{i+1})^3 & \quad \textrm{for} \quad  2 \leq i \leq n-2,
\end{cases}\\
&&\\
\tau(y_1 \rho_1(\rho_i\rho_{i+1})^3 \rho_1 y_1) &=& \big ( \gamma(y_1 \rho_1, \rho_i)\gamma(y_1, \rho_{i+1})\big ) ^3\\
&=&\begin{cases}
\beta_2^{3}& \quad \textrm{for} \quad i=1,\\
(\beta_i^{-1}\beta_{i+1})^3 & \quad \textrm{for} \quad 2\leq i\leq n-2.
\end{cases}
\end{eqnarray*}

Thus, we have the following non-trivial relations in $VL_n^{'}$ 
\begin{eqnarray}
\label{vln comm rel 1} \alpha_2^3&=1,&\\
\label{vln comm rel 2} \beta_2^{3}&=1,&\\
\label{vln comm rel 3} (\alpha_i \alpha_{i+1}^{-1})^3&=1&  \quad \textrm{for} \quad 2\leq i\leq n-2,\\
\label{vln comm rel 4} (\beta_i \beta_{i+1}^{-1})^3&=1& \quad \textrm{for} \quad 2\leq i\leq n-2.
\end{eqnarray}
\item[]

\item The relations  $(\rho_i y_1)^2=1$ for $ 3 \leq i \leq n-1$ give
\begin{eqnarray*} 
\tau(1(\rho_i y_1)^2 1) &=& \alpha_i \delta \beta_i^{-1},\\
&&\\
\tau(y_1(\rho_i y_1)^2 y_1) &=& \beta_i \delta^{-1} \alpha_i^{-1},\\
&&\\
\tau(\rho_1 (\rho_i y_1)^2 \rho_1) &=& \alpha_i^{-1}\beta_i \delta^{-1},\\
&&\\
\tau( y_1 \rho_1 (\rho_i y_1)^2 \rho_1 y_1) &=& \beta_i^{-1} \alpha_i \delta
\end{eqnarray*}
for each $3 \le i \le n-1$. Thus, the non-trivial relations in $VL_n^{'}$ are
\begin{eqnarray}
\label{vln comm rel 9} \beta_i &=& \alpha_i \delta  \quad \textrm{for} \quad   3 \leq i \leq n-1.
\end{eqnarray}
\item[]

\item Finally, we consider the relation $(y_1\rho_1 \rho_2 y_1 \rho_2 \rho_1)^3 = 1$. In this case, we get
\begin{eqnarray*}
 \tau(1(y_1\rho_1 \rho_2 y_1 \rho_2 \rho_1)^3 1)&=& (\gamma(y_1 \rho_1, \rho_2) \gamma(1, \rho_2) )^3 = (\beta_2^{-1} \alpha_2)^3,\\
 & &\\
\tau(y_1(y_1\rho_1 \rho_2 y_1 \rho_2 \rho_1)^3 y_1)&=& (\gamma(\rho_1, \rho_2) \gamma(y_1, \rho_2))^3= (\alpha_2^{-1} \beta_2)^3,\\
 & &\\
\tau(\rho_1(y_1\rho_1 \rho_2 y_1 \rho_2 \rho_1)^3 \rho_1) &=& (\gamma(\rho_1, y_1) \gamma(y_1, \rho_2)\gamma(y_1 \rho_1, y_1) \gamma(\rho_1, \rho_2) )^3= (\delta \beta_2 \delta^{-1} \alpha_2^{-1})^3,\\
 & &\\
\tau(y_1\rho_1(y_1\rho_1 \rho_2 y_1 \rho_2 \rho_1)^3 \rho_1 y_1)& = & (\gamma(y_1\rho_1, y_1) \gamma(1, \rho_2)\gamma(\rho_1, y_1) \gamma(y_1\rho_1, \rho_2) )^3 = ( \delta^{-1} \alpha_2 \delta \beta_2^{-1})^3.
\end{eqnarray*}

Thus, we have the following additional non-trivial relations in $VL_n^{'}$ 
\begin{eqnarray}
\label{vln comm rel 10}  (\alpha_2^{-1} \beta_2)^3 &= &1,\\
\label{vln comm rel 11} (\delta \beta_2 \delta^{-1} \alpha_2^{-1})^3 &= &1.
\end{eqnarray}
\end{enumerate}

We eliminate the generators $\beta_i$ for $3 \leq i \leq n-1$ using the relations \eqref{vln comm rel 9} and set $\beta := \beta_2$. This together with the relations \eqref{vln comm rel 1} through \eqref{vln comm rel 11} gives the following result.

\begin{theorem}\label{commutator-subgroup-vln}
The following hold for the commutator subgroup $VL_n^{'}$ of $VL_n$:
\begin{enumerate}
\item $VL_2^{'} \cong \mathbb{Z}$ and is generated by $(\rho_1 y_1)^2$.

\item $VL_3^{'} \cong \mathbb{Z}_3 * \mathbb{Z}_3 * \mathbb{Z}$ and has a presentation
$$ \big\langle \rho_2\rho_1,  y_1\rho_2\rho_1y_1, (\rho_1 y_1)^2 ~\mid~ (\rho_2\rho_1)^3 = (y_1\rho_2\rho_1y_1)^3=1 \big\rangle.$$

\item For  $n \geq 4$, $VL_n^{'}$ has a presentation with generating set
$$\big\{ \alpha_i, \beta, \delta~\mid~ i=2, 3, \dots, n-1 \big\}$$
and defining relations as follows:
\begin{eqnarray*}
\alpha_2^3&=&1,\\
\alpha_j^2&=&1  \quad \textrm{for} \quad 3\leq j\leq n-1,\\
\beta^{3}&=&1,\\
(\alpha_i\alpha_{i+1}^{-1})^3&=&1  \quad \textrm{for} \quad 2\leq i\leq n-2,\\
(\alpha_i\alpha_{j}^{-1})^2&=&1  \quad \textrm{for} \quad  2 \leq i \leq n-2   \quad \text{and}  \quad  j \geq i+2,\\
(\alpha_j\delta)^{2}&=&1   \quad \textrm{for} \quad 3\leq j\leq n-1,\\
(\beta \delta^{-1}\alpha_3^{-1})^3&=&1,\\
(\beta \delta^{-1}\alpha_j^{-1})^2&=&1  \quad \textrm{for} \quad  4 \leq j\leq n-1,\\
(\alpha_2^{-1} \beta)^3&=&1,\\
(\delta \beta \delta^{-1} \alpha_2^{-1})^3&=&1.
\end{eqnarray*}
\end{enumerate}
\end{theorem}

\begin{proposition}\label{stable-lower-central-series-Ln}
 $L_n^{'}=\gamma_3(L_n)$ and $VL_n^{'} = \gamma_3 (VL_n)$ for $n \geq 3$.
\end{proposition}

\begin{proof}
Consider the short exact sequence
$$
1 \longrightarrow L_n^{'} / \gamma_3\left(L_n\right) \longrightarrow L_n / \gamma_3\left(L_n\right) \longrightarrow \mathbb Z_2 \longrightarrow 1.
$$
Since coset of each $y_i$ maps to the generator of $\mathbb Z_2$,  it follows that there exist $\alpha_i \in L_n^{'}$ such that $y_i=\alpha_i y_1 \mod \gamma_3\left(L_n\right)$ for each $1 \leq i \leq n-1$, where $\alpha_1=1$.  Using the braid relation in $L_n / \gamma_3\left(L_n\right)$, we obtain
\begin{equation}\label{Lnlowercentralseries1}
\alpha_{i}y_1\alpha_{i+1}y_1\alpha_{i}y_1=\alpha_{i+1}y_1\alpha_{i}y_1\alpha_{i+1}y_1  \mod \gamma_3\left(L_n\right)
\end{equation}
for all $1 \leq i \leq n-1$. Since $\alpha_1=1$ and $L_n^{'} / \gamma_3\left(L_n\right)\subseteq \Z(L_n / \gamma_3\left(L_n\right))$,  equation \eqref{Lnlowercentralseries1} gives
$$\alpha_{i}=1 \mod \gamma_3\left(L_n\right)$$
for all $1 \leq i \leq n-1$.  This implies that $L_n / \gamma_3\left(L_n\right)$ is cyclic of order 2, and hence $L_n^{'} / \gamma_3\left(L_n\right)=1$, which is desired. The proof of $VL_n^{'} = \gamma_3 (VL_n)$ is similar to that of Proposition \ref{stable-lower-central-series-tn} with $s_i$ replaced by $y_i$.
\end{proof}

\begin{corollary}\label{vln residually nilpotent}
$L_n$ and $VL_n$ are residually nilpotent if and only if $n =2$. 
\end{corollary}
\medskip

\section{Presentation of the pure virtual triplet group}\label{sec presentation pure virtual triplet group}
Recall that, for $n\geq 2$, the virtual triplet group $VL_n$ has a presentation with generating set $\{ y_1, y_2, \ldots, y_{n-1}, \rho_1, \rho_2, \ldots, \rho_{n-1}\}$ and defining relations as follows:
\begin{eqnarray}
y_i^{2} &=&1 \quad \textrm{for} \quad 1 \le i \le n-1, \label{L1}\\ 
y_iy_{i+1}y_{i} &=& y_{i+1}y_{i}y_{i+1} \quad\textrm{for}\quad 1 \le i \le n-2, \label{L2}\\ 
\rho_i^{2} &=& 1 \quad \textrm{for}\quad 1 \le i \le n-1, \label{vln L3}\\
\rho_i\rho_j &=& \rho_j\rho_i \quad \textrm{for} \quad |i - j| \geq 2, \label{vln  L4}\\
\rho_i\rho_{i+1}\rho_i &=& \rho_{i+1}\rho_i\rho_{i+1} \quad \textrm{for}\quad 1 \le i \le n-2, \label{vln L5}\\
\rho_i y_j &=& y_j\rho_i \quad \textrm{for} \quad |i - j| \geq 2, \label{L6}\\
\rho_i\rho_{i+1} y_i &=& y_{i+1} \rho_i \rho_{i+1} \quad \textrm{for} \quad 1 \le i \le n-2. \label{L7}
\end{eqnarray}
Let $\pi: V L_n \longrightarrow S_n$ denote the natural epimorphism given by $\pi\left(y_i\right)=\pi\left(\rho_i\right)=\tau_i$ for $1 \le i \le n-1$, where $\tau_i$ denotes the transposition $(i, i+1)$ in $S_n$. Then, we have $PV L_n=\ker(\pi)$. Further, the monomorphism $S_n \to V L_n$ given by $\tau_i \mapsto \rho_i$ gives a splitting of the exact sequence

\begin{equation*}\label{SESVLn}
1 \longrightarrow PV L_n \longrightarrow V L_n \stackrel{\pi}{\longrightarrow} S_n \longrightarrow 1.
\end{equation*}

We take the set 
 $$\mathrm{M}_n=\left\{m_{1, i_1} m_{2, i_2} \ldots m_{n-1, i_{n-1}} ~\mid~ m_{k, i_k}=\rho_k \rho_{k-1} \ldots \rho_{i_k+1}\right.\text{ for }1 \leq k \leq n-1 \text{ and } \left.0 \leq i_k<k\right\}$$
 as a Schreier system of coset representatives of $PVL_n$ in $VL_n$. Further, we set $m_{k, k}=1$ for $1 \leq k \leq n-1$. For an element $w \in V L_n$, let $\overline{w}$ denote the unique coset representative of the coset of $w$ in the Schreier set $\mathrm{M}_n$. Then $PVL _n$ is generated by $$\left\{\gamma(\mu, a)~\mid~ \mu \in \mathrm{M}_n~\text{ and } ~a \in\left\{y_1, \ldots, y_{n-1}, \rho_1, \ldots, \rho_{n-1}\right\}\right\}$$
and has defining relations
$$\{\tau\left(\mu r \mu^{-1}\right) ~\mid~ \mu \in \mathrm{M}_n~ \text{and} ~r ~\text{is a defining relation in} ~ VL_n\}.$$
We set
$$
\kappa_{i, i+1} :=y_i \rho_i
$$
for each $1 \leq i \leq n-1$ and
$$
\kappa_{i, j} :=\rho_{j-1} \rho_{j-2} \ldots \rho_{i+1} \kappa_{i, i+1} \rho_{i+1} \ldots \rho_{j-2} \rho_{j-1}
$$
for each $1 \leq i<j \leq n$ and $j \neq i+1$. Let  $\mathcal{S}=\left\{\kappa_{i, j} \mid 1 \leq i<j \leq n\right\}$ and $\mathcal{S} \sqcup \mathcal{S}^{-1}=\left\{\kappa_{i, j}^{\pm 1} \mid \kappa_{i, j} \in \mathcal{S}\right\}$. For ease of notation, we set $\kappa_{j,i}:=\kappa_{i,j}^{-1}$ whenever $i<j$.

\begin{lemma}\label{conj action lemma}
The conjugation action of $S_n \cong \langle \rho_1,\ldots,\rho_{n-1} \rangle$ on $\mathcal{S} \sqcup \mathcal{S}^{-1}$ is given by
$$
\rho_i: \begin{cases}\kappa_{i, i+1} \longleftrightarrow \kappa_{i, i+1}^{-1}, & \\ \kappa_{i, j} \longleftrightarrow \kappa_{i+1, j} & \quad \text {for} \quad i+2 \leq j \leq n, \\ \kappa_{j, i} \longleftrightarrow \kappa_{j, i+1} & \quad \text {for}\quad  1 \leq j<i, \\ \kappa_{k, l} \longleftrightarrow \kappa_{k, l} & \quad \text {otherwise.}
\end{cases}
$$
More precisely, $S_n$ acts on   $\mathcal{S} \sqcup \mathcal{S}^{-1}$ by the rule
\begin{equation*}\label{actionVLn}
\rho\cdot{\kappa_{i,j}}=\kappa_{\rho(i),\rho(j)} 
\end{equation*}
for each $\rho \in S_n$ and $\kappa_{i,j} \in \mathcal{S}$.
\end{lemma}

\begin{proof}
First, we consider action on the generators $\kappa_{i, i+1}$ for $1 \le i \le n-1$.
\begin{enumerate}
\item If $1 \leq k \leq i-2$ or $i+2 \leq k \leq n-1$, then $\rho_k \kappa_{i, i+1} \rho_k=\kappa_{i, i+1}$.
\item If $k=i-1$, then
\begin{eqnarray*}
\rho_k \kappa_{i, i+1} \rho_k & =&  \rho_{i-1} y_i \rho_i \rho_{i-1} \\
& =& \rho_{i-1} y_i \rho_{i-1} (\rho_{i-1} \rho_i \rho_{i-1}) \\
& =& \rho_{i-1} (y_i \rho_{i-1} \rho_{i}) \rho_{i-1} \rho_{i} \\
& =& \rho_i\left(y_{i-1} \rho_{i-1}\right) \rho_i \\
& =& \kappa_{i-1, i+1}.
\end{eqnarray*}
\item If $k=i$, then $\rho_k \kappa_{i, i+1} \rho_k=\kappa_{i, i+1}^{-1}$.
\item If $k=i+1$, then $\rho_k \kappa_{i, i+1} \rho_k=\kappa_{i, i+2}$.
\end{enumerate}

Next, we consider action on $\kappa_{i, j}$ for  $1 \leq i<j \leq n$ with $j \neq i+1$.
\begin{enumerate}
\item If $1 \leq k \leq i-2$ or $j+1 \leq k \leq n-1$, then $\rho_k \kappa_{i, j} \rho_k=\kappa_{i, j}$.
\item For $k=i-1$, we have 
\begin{eqnarray*}
\rho_{i-1} \kappa_{i, j} \rho_{i-1} & =& \rho_{i-1} \rho_{j-1} \rho_{j-2} \ldots \rho_{i+1} \kappa_{i, i+1} \rho_{i+1} \ldots \rho_{j-2} \rho_{j-1} \rho_{i-1} \\
& =& \rho_{j-1} \rho_{j-2} \ldots \rho_{i+1} (\rho_{i-1} \kappa_{i, i+1} \rho_{i-1}) \rho_{i+1} \ldots \rho_{j-2} \rho_{j-1} \\
& =& \rho_{j-1} \rho_{j-2} \ldots \rho_{i+1} \rho_i\left(y_{i-1} \rho_{i-1}\right) \rho_i \rho_{i+1} \ldots \rho_{j-2} \rho_{j-1} \\
& =& \kappa_{i-1, j}.
\end{eqnarray*}

\item For $k=i$, we have 
\begin{eqnarray*}
\rho_i \kappa_{i, j} \rho_i & =& \rho_i \rho_{j-1} \rho_{j-2} \ldots \rho_{i+1} \kappa_{i, i+1} \rho_{i+1} \ldots \rho_{j-2} \rho_{j-1} \rho_i \\
& =& \rho_{j-1} \rho_{j-2} \ldots \rho_i \rho_{i+1} \kappa_{i, i+1} \rho_{i+1} \rho_i \ldots \rho_{j-2} \rho_{j-1} \\
& =& \rho_{j-1} \rho_{j-2} \ldots (\rho_i \rho_{i+1} y_i) \rho_i \rho_{i+1} \rho_i \ldots \rho_{j-2} \rho_{j-1} \\
& =& \rho_{j-1} \rho_{j-2} \ldots \rho_{i+2} y_{i+1} (\rho_i \rho_{i+1} \rho_i \rho_{i+1} \rho_i) \ldots \rho_{j-2} \rho_{j-1} \\
& =& \rho_{j-1} \rho_{j-2} \ldots \rho_{i+2}\left(y_{i+1} \rho_{i+1}\right) \rho_{i+2} \ldots \rho_{j-2} \rho_{j-1} \\
& =& \kappa_{i+1, j}.
\end{eqnarray*}

\item If $i+1 \leq k \leq j-2$, then

\begin{eqnarray*}
\rho_k \kappa_{i, j} \rho_k & =& \rho_k \rho_{j-1} \ldots \rho_{k+1} \rho_k \ldots \rho_{i+1} \kappa_{i, i+1} \rho_{i+1} \ldots \rho_k \rho_{k+1} \ldots \rho_{j-1} \rho_k \\
& =& \rho_{j-1} \ldots \rho_k \rho_{k+1} \rho_k \ldots \rho_{i+1} \kappa_{i, i+1} \rho_{i+1} \ldots \rho_k \rho_{k+1} \rho_k \ldots \rho_{j-1} \\
& =& \rho_{j-1} \ldots \rho_{k+1} \rho_k \rho_{k+1} \rho_{k-1} \ldots \rho_{i+1} \kappa_{i, i+1} \rho_{i+1} \ldots \rho_{k-1} \rho_{k+1} \rho_k \rho_{k+1} \ldots \rho_{j-1} \\
& =& \rho_{j-1} \ldots \rho_{k+1} \rho_k \rho_{k-1} \ldots \rho_{i+1} \kappa_{i, i+1} \rho_{i+1} \ldots \rho_{k-1} \rho_k \rho_{k+1} \ldots \rho_{j-1} \\
& =& \kappa_{i, j} .
\end{eqnarray*}
\item If $k=j-1$, then $\rho_k \kappa_{i, j} \rho_k=\kappa_{i, j-1}$.
\item Finally, if $k=j$, then $\rho_k \kappa_{i, j} \rho_k=\kappa_{i, j+1}$.
\end{enumerate}
This completes the proof of the lemma.
\end{proof}

\begin{theorem}\label{presentation pure virtual triplet group}
 The pure virtual triplet group $P V L_n$ has a presentation with generating set $\mathcal{S}=\left\{\kappa_{i, j} \mid 1 \leq i < j \leq n\right\}$ and defining relations
$$
\kappa_{i , j} \kappa_{i , k} \kappa_{j , k}=\kappa_{j , k} \kappa_{i , k} \kappa_{i , j},
$$ 
where $1\leq i < j <k \leq n.$ 
\end{theorem}

\begin{proof}
We begin by observing that
$$ \overline{\alpha}  =\alpha \quad \textrm{and} \quad  \overline{\alpha_1 y_{i_1} \ldots \alpha_k y_{i_k}}  =\alpha_1 \rho_{i_1} \ldots \alpha_k \rho_{i_k}$$
for words $\alpha$ and $\alpha_j$ in the generators $\left\{\rho_1, \ldots, \rho_{n-1}\right\}$. Thus, we have
$$
\gamma\left(\mu, \rho_i\right)=\left(\mu \rho_i\right)\left(\mu \rho_i\right)^{-1}=1
$$
and
$$
\gamma\left(\mu, y_i\right)=\left(\mu y_i\right)\left(\mu \rho_i\right)^{-1}=\mu y_i \rho_i \mu^{-1}=\mu \kappa_{i, i+1} \mu^{-1}
$$
for each $\mu \in \mathrm{M}_n$ and $1 \le i \le n-1$. It follows from Lemma \ref{conj action lemma} that each $\gamma\left(\mu, y_i\right)$ lie in $\mathcal{S} \sqcup \mathcal{S}^{-1}$.  Conversely, if $\kappa_{i, j} \in \mathcal{S}$ is an arbitrary element, then conjugation by $\left(\rho_{i-1} \rho_{i-2} \ldots \rho_2 \rho_1\right)\left(\rho_{j-1} \rho_{j-2} \ldots \rho_3 \rho_2\right)$ maps $\kappa_{1,2}$ to $\kappa_{i, j}$. Similarly, conjugation by the element $\left(\rho_{i-1} \rho_{i-2} \ldots \rho_3\rho_2 \right)\left(\rho_{j-1} \rho_{j-2} \ldots  \rho_2 \rho_1\right)$ maps $\kappa_{1,2}$ to $\kappa_{i, j}^{-1} ~(=\kappa_{j,i})$. Thus,  $\mathcal{S} \sqcup \mathcal{S}^{-1}$, and consequently $\mathcal{S}$ generates $PVL_n$.
\par

Next, we determine the defining relations in $P V L_n$. Consider 
an element $\mu=\rho_{i_1} \rho_{i_2} \ldots \rho_{i_k} \in \mathrm{M}_n.$ Since $\gamma\left(\mu, \rho_i\right)=1$ for all $i$, no non-trivial relations for $P V L_n$ can be obtained from the relations \eqref{vln L3}, \eqref{vln L4} and \eqref{vln L5} of $VL_n$. We now consider the remaining relations of $VL_n$.
\begin{enumerate}
\item First, consider the relations $y_i^2=1$ for $1 \leq i \leq n-1$. Then we have
\begin{eqnarray*}
\tau\left(\mu y_i^2 \mu^{-1}\right) & = &\gamma\left(\rho_{i_1} \ldots \rho_{i_k} y_i y_i \rho_{i_k} \ldots \rho_{i_1}\right) \\
& = &\gamma\left(1, \rho_{i_1}\right) \gamma\left(\overline{\rho_{i_1}}, \rho_{i_2}\right) \ldots \gamma\left(\bar{\mu}, y_i\right) \gamma\left(\overline{\mu y_i}, y_i\right) \ldots \gamma\left(\overline{\mu y_i y_i \rho_{i_k} \ldots \rho_{i_2}}, \rho_{i_1}\right) \\
& = &\gamma\left(\bar{\mu}, y_i\right) \gamma\left(\overline{\mu y_i}, y_i\right) \\
& = &\gamma\left(\mu, y_i\right) \gamma\left(\mu \rho_i, y_i\right) \\
& = & \gamma\left(\mu, y_i\right) \left(\mu \rho_i y_i \mu^{-1}\right) \\
& = &\gamma\left(\mu, y_i\right)  \gamma\left(\mu, y_i\right) ^{-1},
\end{eqnarray*}
which does not yield any non-trivial relation in $P V L_n$.

\item For the relations $(y_iy_{i+1})^3=1$ for $1 \leq i \leq n-2$, we have
\begin{eqnarray*}
\tau\left(\mu (y_iy_{i+1})^3 \mu^{-1}\right) & = &\gamma\left(\rho_{i_1} \ldots \rho_{i_k} y_iy_{i+1}y_iy_{i+1}y_iy_{i+1} \rho_{i_k} \ldots \rho_{i_1}\right) \\
& = &\gamma\left(\bar{\mu}, y_i\right) \gamma\left(\overline{\mu y_i}, y_{i+1}\right)\gamma\left(\overline{\mu y_iy_{i+1}}, y_i\right)\\
& & \gamma\left(\overline{\mu y_iy_{i+1}y_i}, y_{i+1}\right) \gamma\left(\overline{\mu y_iy_{i+1}y_iy_{i+1}}, y_{i}\right)\gamma\left(\overline{\mu y_iy_{i+1}y_iy_{i+1}y_{i}}, y_{i+1}\right) \\
& = &\gamma\left(\mu, y_i\right) \gamma\left(\mu \rho_i, y_{i+1}\right)\gamma\left(\mu \rho_i\rho_{i+1}, y_i\right)\\
& & \gamma\left(\mu \rho_i \rho_{i+1} \rho_i, y_{i+1}\right) \gamma\left(\mu \rho_i \rho_{i+1} \rho_i \rho_{i+1}, y_{i}\right)\gamma\left(\mu \rho_i \rho_{i+1} \rho_i \rho_{i+1} \rho_{i}, y_{i+1}\right) \\
& = &  \left(\mu \kappa_{i, i+1} \mu^{-1}\right) \left(\mu \kappa_{i, i+2} \mu^{-1}\right) \left(\mu \kappa_{i+1, i+2} \mu^{-1}\right)\\
& & \left(\mu \kappa_{i, i+1} \mu^{-1}\right)^{-1} \left(\mu \kappa_{i, i+2}\mu^{-1}\right)^{-1}  \left(\mu \kappa_{i+1, i+2} \mu^{-1}\right)^{-1}\\
& = &  \kappa_{\mu(i), \mu(i+1)} \kappa_{\mu(i), \mu(i+2)} \kappa_{\mu(i+1), \mu(i+2)}\kappa_{\mu(i), \mu(i+1)}^{-1} \kappa_{\mu(i),\mu(i+2)}^{-1} \kappa_{\mu(i+1), \mu(i+2)}^{-1}.
\end{eqnarray*}

This gives non-trivial relations 
\begin{equation}\label{P3}
\kappa_{\mu(i), \mu(i+1)} \kappa_{\mu(i), \mu(i+2)} \kappa_{\mu(i+1), \mu(i+2)}= \kappa_{\mu(i+1), \mu(i+2)} \kappa_{\mu(i),\mu(i+2)} \kappa_{\mu(i), \mu(i+1)}
\end{equation}
in $PVL_n$. Since $S_n$ action on $\{1,2,\ldots,n\}$ is triply-transitive, \eqref{P3} gives the non-trivial relations
\begin{equation}\label{P1}
\kappa_{i , j} \kappa_{i , k} \kappa_{j , k}=\kappa_{j , k} \kappa_{i , k} \kappa_{i , j},
\end{equation}
where $i,j,k$ are all distinct. We claim that the relations \eqref{P1} are consequences of the relations
\begin{equation}\label{P2}
\kappa_{p, q} \kappa_{p , r} \kappa_{q , r}=\kappa_{q , r} \kappa_{p , r} \kappa_{p , q},
\end{equation}
where $1\leq p < q < r \leq n$.
\medskip
Note that, given three distinct integers $i,j,k$, one of the following holds:
\begin{enumerate}
 \begin{multicols}{3}
\item $i < j < k$
\item $i < k < j $
\item $k < j <i $
\item $k < i < j$
\item $j < i < k$
\item $j < k < i$
 \end{multicols}
\end{enumerate}
\par

Nothing needs to be done in Case (a). In Case (b), the relation $\eqref{P1}$ becomes 
$$\kappa_{i , j} \kappa_{i , k} \kappa_{k , j}^{-1}=\kappa_{k , j}^{-1} \kappa_{i , k} \kappa_{i , j}.$$
Rewriting it gives
$$\kappa_{k, j}\kappa_{i , j} \kappa_{i , k} = \kappa_{i , k} \kappa_{i , j} \kappa_{k , j},$$
which is the relation $\eqref{P2}$. In Case (c),  the relation $(\ref{P1})$ becomes
$$\kappa_{j , i}^{-1} \kappa_{k , i}^{-1} \kappa_{k , j}^{-1}=\kappa_{k , j}^{-1} \kappa_{k , i}^{-1} \kappa_{j , i}^{-1}.$$
Taking inverses give
$$\kappa_{k, j}\kappa_{k , i} \kappa_{j , i} = \kappa_{j , i} \kappa_{k , i} \kappa_{k , j},$$
which is the relation $\eqref{P2}$. The remaining cases can sorted out in a similar manner.

\item In case of the relations $\left(y_i \rho_j\right)^2=1$ for $|i-j| \ge 2$, we obtain

\begin{eqnarray*}
\tau\left(\mu y_i \rho_j y_i \rho_j \mu^{-1}\right) & =& \gamma\left(\bar{\mu}, y_i\right) \gamma\left(\overline{\mu y_i}, \rho_j\right) \gamma\left(\overline{\mu y_i \rho_j}, y_i\right) \gamma\left(\overline{\mu y_i \rho_j y_i}, \rho_j\right) \\
 & =& \gamma\left(\bar{\mu}, y_i\right) \gamma\left(\overline{\mu y_i \rho_j}, y_i\right) \\
 & =& \gamma\left(\mu, y_i\right) \gamma\left(\mu \rho_i \rho_j, y_i\right) \\
 & =& \gamma\left(\mu, y_i\right) \left(\mu \rho_i \rho_j y_i \rho_i \rho_j \rho_i \mu^{-1}\right) \\
 & =& \gamma\left(\mu, y_i\right)  \left(\mu \rho_i y_i \mu^{-1}\right) \\
 & =& \gamma\left(\mu, y_i\right) \left(\mu y_i \rho_i \mu^{-1}\right)^{-1}\\
 & =& \gamma\left(\mu, y_i\right) \gamma\left(\mu, y_i\right)^{-1},
\end{eqnarray*}

which does not yield any non-trivial relation in $P V L_n$.

\item Finally, consider the relations $\rho_i y_{i+1} \rho_i \rho_{i+1} y_i \rho_{i+1}=1$ for $1 \leq i \leq n-2$. Computing

\begin{eqnarray*}
\tau\left(\mu \rho_i y_{i+1} \rho_i \rho_{i+1} y_i \rho_{i+1} \mu^{-1}\right) & = &\gamma\left(\overline{\mu \rho_i}, y_{i+1}\right) \gamma\left(\overline{\mu \rho_i y_{i+1} \rho_i \rho_{i+1}}, y_i\right) \\
 & =& \gamma\left(\mu \rho_i, y_{i+1}\right) \gamma\left(\mu \rho_i \rho_{i+1} \rho_i \rho_{i+1}, y_i\right) \\
 & =& \gamma\left(\mu \rho_i, y_{i+1}\right)\left(\mu \rho_{i+1} \rho_i y_i \rho_{i+1} \mu^{-1}\right) \\
 & =& \gamma\left(\mu \rho_i, y_{i+1}\right) \left(\mu \rho_{i+1} \rho_i \rho_{i+1} \rho_{i+1} y_i \rho_{i+1} \mu^{-1}\right) \\
 & =& \gamma\left(\mu \rho_i, y_{i+1}\right) \left(\mu \rho_i  \rho_{i+1} \rho_i \rho_iy_{i+1} \rho_i \mu^{-1}\right) \\
 & =& \gamma\left(\mu \rho_i, y_{i+1}\right) \left(\mu \rho_i \rho_{i+1} y_{i+1} \rho_i \mu^{-1}\right) \\
 & =& \gamma\left(\mu \rho_i, y_{i+1}\right) \gamma\left(\mu \rho_i, y_{i+1}\right)^{-1},
\end{eqnarray*}
gives only trivial relations in $P V L_n$. 
\end{enumerate}
Thus, the only non-trivial relations amongst elements of $\mathcal{S}$ are of the form $\kappa_{i , j} \kappa_{i , k} \kappa_{j , k}=\kappa_{j , k} \kappa_{i , k} \kappa_{i , j}$, where $1\leq i < j <k \leq n$. This completes the proof of the theorem.
\end{proof}
\medskip

\section{Crystallographic quotients of virtual triplet groups}
A closed subgroup $H$ of a Hausdorff topological group $G$ is said to be {\it uniform} if $G/H$ is a compact space.  A discrete and uniform subgroup $G$ of $~\mathbb{R}^n \rtimes \Oo(n, \mathbb{R})$ is called a {\it crystallographic group} of dimension $n$. If in addition $G$ is torsion free, then it is called a {\it Bieberbach group} of dimension $n$. The following characterisation of crystallographic groups is well-known \cite[Lemma 8]{MR3595797}.

\begin{lemma}
 A group $G$ is a crystallographic group if and only if there is an integer $n$, a finite group $H$ and a short exact sequence
$$
0 \longrightarrow \mathbb{Z}^n \longrightarrow G \stackrel{\eta}{\longrightarrow} H \longrightarrow 1
$$
such that the integral representation $\Theta : H \longrightarrow \Aut\left(\mathbb{Z}^n\right)$ defined by $\Theta(h)(x)=z x z^{-1}$ is faithful, where  $h \in H$, $x \in \mathbb{Z}^n$ and $z \in G$ is such that $\eta(z)=h$. 
\end{lemma}

The group $H$ is called the {\it holonomy group} of $G$, the integer $n$ is called the {\it dimension} of $G$, and $\Theta : H \longrightarrow \Aut\left(\mathbb{Z}^n\right)$ is called the {\it holonomy representation} of $G$. It is known that any
finite group is the holonomy group of some flat manifold \cite[Theorem III.5.2]{MR0862114}. Furthermore, there is a correspondence between the class of Bieberbach groups and the class of compact flat Riemannian manifolds \cite[Theorem 2.1.1]{MR1482520}. Crystallographic quotients of virtual braid and virtual twin groups have been considered in \cite[Theorem 2.4]{Ocampo-Santos-2021} and \cite[Theorem 3.5]{Ocampo-Santos-2021}, respectively.

\begin{theorem}
 For $n \geq 2$, there is a split short exact sequence
$$
1 \longrightarrow \mathbb{Z}^{n(n-1)} \longrightarrow V B_n / V P_n^{'} \longrightarrow S_n \longrightarrow 1
$$
such that the group $V B_n / V P_n^{'}$ is a crystallographic group of dimension $n(n-1)$.
\end{theorem}

\begin{theorem}
 For $n \geq 2$, there is a split short exact sequence

$$
1 \longrightarrow \mathbb{Z}^{n(n-1) / 2} \longrightarrow VT_n /PVT_n^{'} \longrightarrow S_n \longrightarrow 1
$$
such that the group $VT_n /PVT_n^{'}$ is a crystallographic group of dimension $n(n-1)/2$.
\end{theorem}

We prove the following result for virtual triplet groups.

\begin{theorem}\label{cryst. quotient virtual triplet mod pure}
 For $n \geq 2$, there is a split short exact sequence
$$
1 \longrightarrow \mathbb{Z}^{n(n-1)/2} \longrightarrow V L_n / PV L_n^{'} \longrightarrow S_n \longrightarrow 1
$$
such that the group $VL_n / PVL_n^{'}$ is a crystallographic group of dimension $n(n-1)/2$.
\end{theorem}

\begin{proof}
Notice that $PV L_n / PV L_n^{'} \cong \mathbb{Z}^{n(n-1)/2}$. The split short exact sequence $$1 \to PVL_n \to VL_n \to S_n \to 1$$ induces the desired split short exact sequence
\begin{equation*}\label{SESVTn2}
1 \longrightarrow \mathbb{Z}^{n(n-1)/2} \longrightarrow VL_n/PVL_n^{'} \stackrel{\pi}{\longrightarrow} S_n \longrightarrow 1.
\end{equation*}
Let $\Theta: S_n\to \Aut(PVL_n / PVL_n^{'}))$ be the action of $S_n$ on $PVL_n / PVL_n^{'}$, which is induced by the action given in Lemma \ref{actionVLn}. That is, for $\rho \in S_n$ and $\kappa_{i, j} \in PVL_n$, we have $\Theta(\rho) (\kappa_{i, j})=\kappa_{\rho(i), \rho(j)} \mod PVL_n^{'}$. Thus,
$\kappa_{\rho(i), \rho(j)} = \kappa_{i, j} \mod PVL_n^{'}$ for all $1 \le i< j \le n$ if and only if  $\rho=1$.  Hence, the holonomy representation is faithful and $VL_n / PVL_n^{'}$ is a crystallographic group.
\end{proof} 
\medskip

Next, we determine torsion in the crystallographic quotient considered above. Since $VL_n =PVL_n \rtimes S_n$, any $\beta \in V L_n$ can be written uniquely as $\beta=w \theta$ for some $w \in PVL_n $ and $\theta \in S_n$. The element $\beta$ acts on the set $\left\{\kappa_{i, j} \mid 1 \leq i \ne j \leq n\right\}$ via the action of its image $\pi(\beta)=\theta$ (see Lemma \ref{conj action lemma}).  We denote the orbit of an element $\kappa_{i, j}$ under this action by $\mathcal{O}_\theta\left(\kappa_{i, j}\right)$. The following theorem is an analogue of similar results for virtual braid groups and virtual twin groups \cite[Theorem 2.7 and Theorem 3.6]{Ocampo-Santos-2021}.

\begin{theorem}\label{cry-torsion-Ln}
For each $2 \leq t \leq n$, let $1 \leq r_1<r_2<\cdots<r_{t-1} \leq n$ be a sequence of consecutive integers. Let $\pi\left(\rho_{r_1} \rho_{r_2} \ldots \rho_{r_{t-1}}\right)=\theta$ and $\mathcal{T}_\theta$ denote a set of representatives of orbits of the action of $\rho_{r_1} \rho_{r_2} \ldots \rho_{r_{t-1}}$ on $\left\{\kappa_{i, j} \mid 1 \leq i \ne j \leq n\right\}$. Then the coset of the element 
$$( \prod_{1 \leq i<j \leq n} \kappa_{i, j}^{a_{i, j}}) \left(\rho_{r_1} \rho_{r_2} \ldots \rho_{r_{t-1}}\right)$$ has order $t$ in $VL_n / PVL_n^{'}$ if and only if
$$
a_{r,s}+a_{\theta^{-1}(r), \theta^{-1}(s)} + \cdots  +  a_{\theta^{-(t-1)}(r), \theta^{-(t-1)}(s)}=0
$$
for all $\kappa_{r, s} \in \mathcal{T}_\theta$. Here, $a_{i, j} \in \mathbb{Z}$ and $a_{\theta^\ell(r), \theta^\ell(s)} :=-a_{\theta^\ell(s), \theta^\ell(r)}$ whenever $1 \leq \theta^\ell(s)<\theta^\ell(r) \leq n$.
\end{theorem}

\begin{proof}
We assume that $r_i=i$ for all $1 \le i \le t-1$, and the other cases can be proved in a similar way. Note that, for each $2 \leq t \leq n$, the element $\rho_1 \rho_2 \ldots \rho_{t-1}$ has order $t$ in $VL_n$. We will analyse the conditions on $a_{i, j} \in \mathbb{Z}$ under which the coset of the element $(\prod_{1 \leq i<j \leq n} \kappa_{i, j}^{a_{i, j}}) (\rho_1 \rho_2 \ldots \rho_{t-1})$ has order $t$ in $V L_n / PVL_n^{'}$. To proceed, let $\pi\left(\rho_1 \rho_2 \ldots \rho_{t-1}\right)=\theta$ and $\mathcal{T}_\theta=\left\{\kappa_{r_1, s_1}, \kappa_{r_2, s_2}, \ldots, \kappa_{r_m, s_m}\right\}$ a set of representatives of orbits of the action of $\rho_1 \rho_2 \ldots \rho_{t-1}$ on $\left\{\kappa_{i, j} \mid 1 \leq i \ne j \leq n\right\}$. Then, we have

\begin{eqnarray*}
& & \left( (\prod_{1 \leq i<j \leq n} \kappa_{i, j}^{a_{i, j}} ) \left(\rho_1 \rho_2 \ldots \rho_{t-1}\right)\right)^t \\
&=& (\prod_{1 \leq i<j \leq n} \kappa_{i, j}^{a_{i, j}})\left(\rho_1 \rho_2 \ldots \rho_{t-1}\right) \cdots (\prod_{1 \leq i<j \leq n} \kappa_{i, j}^{a_{i, j}}) \left(\rho_1 \rho_2 \ldots \rho_{t-1}\right) \\
&=& (\prod_{1 \leq i<j \leq n} \kappa_{i, j}^{a_{i, j}}) \left( \left(\rho_1 \ldots \rho_{t-1}\right) (\prod_{1 \leq i<j \leq n} \kappa_{i, j}^{a_{i, j}}) \left(\rho_1 \ldots \rho_{t-1}\right)^{-1} \right) \\
& & \left(\left(\rho_1 \ldots \rho_{t-1}\right)^2 (\prod_{1 \leq i<j \leq n} \kappa_{i, j}^{a_{i, j}})  \left(\rho_1 \ldots \rho_{t-1}\right)^{-2} \right) \cdots\\
& &\cdots \left( \left(\rho_1 \ldots \rho_{t-1}\right)^{t-1} (\prod_{1 \leq i<j \leq n} \kappa_{i, j}^{a_{i, j}}) \left(\rho_1 \ldots \rho_{t-1}\right)^{-(t-1)} \right) \left(\rho_1 \ldots \rho_{t-1}\right)^t \\
&=& (\prod_{1 \leq i<j \leq n} \kappa_{i, j}^{a_{i, j}})~ \left( \theta \cdot (\prod_{1 \leq i<j \leq n} \kappa_{i, j}^{a_{i, j}})\right)~ \left( \theta^2 \cdot (\prod_{1 \leq i<j \leq n} \kappa_{i, j}^{a_{i, j}} )\right)  \cdots \\
& & \cdots \left( \theta^{t-1} \cdot (\prod_{1 \leq i<j \leq n} \kappa_{i, j}^{a_{i, j}})\right) ~(\rho_1 \ldots \rho_{t-1} )^t \\
&=& (\prod_{1 \leq i<j \leq n} \kappa_{i, j}^{a_{i, j}}) (\prod_{1 \leq i<j \leq n} \kappa_{\theta(i), \theta(j)}^{a_{i, j}} ) (\prod_{1 \leq i<j \leq n} \kappa_{\theta^2(i), \theta^2(j)}^{a_{i, j}}) \cdots ( \prod_{1 \leq i<j \leq n} \kappa_{\theta^{t-1}(i), \theta^{t-1}(j)}^{a_{i, j}}).\\
& =& \prod_{1 \leq i<j \leq n} \left(\kappa_{i, j}^{a_{i, j}}\kappa_{i, j}^{a_{\theta^{-1}(i), \theta^{-1}(j)}}\kappa_{i, j}^{a_{\theta^{-2}(i), \theta^{-2}(j)}} \cdots \kappa_{i, j}^{a_{\theta^{-(t-1)}(i), \theta^{-(t-1)}(j)}}\right) \mod PVL_n^{'}.
\end{eqnarray*}

The last expression implies that the total exponent sum of a generator $\kappa_{i, j}$ (with $1 \le i<j \le n$) is the same as that of $\kappa_{\theta^\ell(i), \theta^\ell(j))}$ for each $0 \le \ell \le t-1$. Thus, using $\mathcal{T}_\theta$, we see that 
$$\left( (\prod_{1 \leq i<j \leq n} \kappa_{i, j}^{a_{i, j}}) (\rho_1 \rho_2 \ldots \rho_{t-1}) \right)^t=1 \mod PVL_n^{'}$$ 
if and only if the following system of integer equations has a solution
 $$
\begin{cases}
    a_{r_1, s_1}+a_{\theta^{-1}\left(r_1\right), \theta^{-1}\left(s_1\right)}+\cdots+ a_{\theta^{-(t-1)}\left(r_1\right), \theta^{-(t-1)}\left(s_1\right)}=0,& \\
    a_{r_2, s_2}+a_{\theta^{-1}\left(r_2\right), \theta^{-1}\left(s_2\right)}+\cdots+a_{\theta^{-(t-1)}\left(r_2\right), \theta^{-(t-1)}\left(s_2\right)}=0, & \\
    \vdots \quad \quad \quad \vdots \quad \quad \quad \vdots \quad \quad \quad \vdots \quad \quad \quad \vdots \quad \quad \quad \vdots \quad \quad \quad \vdots& \\
     a_{r_m, s_m}+a_{\theta^{-1}\left(r_m\right), \theta^{-1}\left(s_m\right)}+\cdots+ a_{\theta^{-(t-1)}\left(r_m\right), \theta^{-(t-1)}\left(s_m\right)}=0.&
	\end{cases}
$$
This completes the proof.
\end{proof}

\begin{corollary}\label{torsion in vln mod pvln}
Let $m_1, m_2, \ldots, m_r$ be positive integers (not necessarily distinct) each greater than 1 such that $\sum_{i=1}^r m_i \leq n$. Then $VL_n / PVL_n^{'}$ admits infinitely many elements of order $\operatorname{lcm}\left(m_1, \ldots, m_r \right)$. Further, there exists such an element whose corresponding permutation has cycle type $\left(m_1, m_2, \ldots, m_r \right)$.
\end{corollary}

\begin{proof}
By Theorem \ref{cry-torsion-Ln}, there exist elements $w_k \in VL_n$ whose cosets have order $m_k$ in $VL_n / PVL_n^{'}$ for each $1 \le k \le r$, where 
\begin{small}
$$w_k=\begin{cases}
(\prod_{\kappa_{i, j} \in \mathcal{O}_{\theta_1}\left(\kappa_{1,2}\right)} \kappa_{i, j}^{a_{i, j}})~( \rho_1 \cdots \rho_{m_1-1} )& \quad \text{for} \quad  k=1,\\
\\
(\prod_{\kappa_{i, j} \in \mathcal{O}_{\theta_k}\left(\kappa_{r_k, s_k}\right)} \kappa_{i, j}^{a_{i, j}} )~(\rho_{\sum_{l=1}^{k-1} m_l+1} \cdots \rho_{\sum_{l=1}^k m_l-1}) & \quad \text{for} \quad 2 \leq k \leq r,
\end{cases}$$
\end{small}
$a_{i, j} \in \mathbb{Z}$, $r_k=\sum_{l=1}^{k-1} m_l+1, s_k=\sum_{l=1}^{k-1} m_l+$ 2 and $\pi\left(w_k\right)=\theta_k$. Since $\rho_{\sum_{l=1}^{k_i-1} m_l+1} \cdots \rho_{\sum_{l=1}^{k_i} m_l-1}$ and $\rho_{\sum_{l=1}^{k_j-1} m_l+1} \cdots \rho_{\sum_{l=1}^{k_j} m_l-1}$ commute in $VL_n / PVL_n^{'}$, it follows that the element $w_1 w_2 \cdots w_r$ has order $\operatorname{lcm}\left(m_1, m_2, \ldots, m_r \right)$ in $VL_n / PVL_n^{'}$.  Further, $\pi\left(w_1 w_2 \cdots w_r \right)$ is an element of $S_n$ of cycle type $\left(m_1, m_2, \ldots, m_r \right)$.
\end{proof}

\begin{remark}
The preceding corollary shows that  $VL_n / PVL_n^{'}$  is not a Bieberbach group. In the spirit of \cite[Question 3.2]{MR4445569}, it is interesting to determine Bieberbach subgroups of $VL_n/PVL_n^{'}$.
\end{remark}

We conclude by tabulating the results for braid groups, twin groups, triplet groups and their virtual counterparts.
\medskip
\begin{Small}
\begin{center}\label{table1}
\begin{tabular}{|c||c|c|c|c|}
  \hline
 & $B_n$ & $T_n$  & $L_n$ \\
&&&\\
   \hline
      \hline
\textrm{Commutator } & Finite presentation  & Finite presentation& Finite presentation \\
subgroup & known for $n \ge2$ \cite{MR0251712}. & known for  $n \ge2$ \cite{MR3943376}.  &  known for  $n \ge2$ \cite{Bourbaki, MR4270786}. \\
&&&\\
  \hline
\textrm{Pure }  & Finite presentation  & Finite presentation & It is a free group of \\
subgroup & known for  $n \ge2$ \cite{MR0019087}. & known for $2 \le n \le 6$ & finite rank for $n \ge 4$ \\
 && \cite{MR4027588, MR4170471, MR4079623}. & \cite{MR1386845, KumarNaikSingh1}. \\
 &&&\\
  \hline
   \textrm{Crystallographic }  & $B_n/P_n^{'}$ is crystallographic&  $T_n/PT_n^{'}$ is crystallographic & $L_n/PL_n^{'}$ is crystallographic\\
quotient&   for $n \ge2$ \cite{MR3595797}. & for $n \ge4$ \cite{KumarNaikSingh1}. & for $n \ge4$ \cite{KumarNaikSingh1}. \\
&&&\\
  \hline
\end{tabular}\\
Table 1.
\end{center}
\end{Small}
\medskip

\begin{Small}
\begin{center}\label{table1}
\begin{tabular}{|c||c|c|c|c|}
  \hline
 & $VB_n$ & $VT_n$  & $VL_n$ \\
&&&\\
   \hline
      \hline
\textrm{Commutator } & Finite generation & Theorem \ref{commutator-subgroup-vtn}. & Theorem \ref{commutator-subgroup-vln}.\\
subgroup& known for  $n \ge 4$ \cite{MR3955820}. &  & \\
& Finite presentation &&\\
& unknown for $n\ge4$. &&\\
&&&\\
  \hline
\textrm{Pure }  & Finite presentation  & Finite presentation & Theorem \ref{presentation pure virtual triplet group}.  \\
subgroup & known for  $n \ge2$ \cite{MR2128039}. & known for $n \ge 2$  \cite{NaikNandaSingh2}. &  \\
 &&& \\
 &&&\\
  \hline
   \textrm{Crystallographic }  & $V B_n / V P_n^{'}$ is crystallographic & $VT_n /PVT_n^{'}$  is crystallographic  & Theorem \ref{cryst. quotient virtual triplet mod pure}.\\
quotient& for $n \ge 2$ \cite{Ocampo-Santos-2021}. & for $n \ge 2$ \cite{Ocampo-Santos-2021}.& \\
&&&\\
  \hline
\end{tabular}\\
Table 2.
\end{center}
\end{Small}

\begin{ack}
Pravin Kumar is supported by the PMRF fellowship at IISER Mohali. Neha Nanda has received funding from European Union's Horizon Europe research and innovation programme under the Marie Sklodowska Curie grant agreement no. 101066588. Mahender Singh is supported by the Swarna Jayanti Fellowship grants DST/SJF/MSA-02/2018-19 and SB/SJF/2019-20. 
\end{ack}

\end{document}